\documentclass{article}
\usepackage[utf8]{inputenc}
\usepackage{amsmath,amssymb,amsthm,mathtools}
\usepackage{booktabs,longtable,array}
\usepackage{enumitem}
\usepackage[colorlinks=true]{hyperref}
\hypersetup{pdftitle={Every graph with no K7= minor is 6-colorable}}

\newtheorem{theorem}{Theorem}[section]
\newtheorem{lemma}[theorem]{Lemma}

\newtheorem{corollary}[theorem]{Corollary}
\newtheorem{observation}[theorem]{Observation}
\newtheorem{conjecture}[theorem]{Conjecture}

\theoremstyle{definition}

\newcommand{\target}{K_7^{=}}
\newcommand{\rhoFour}{\rho_4}
\newcommand{\bd}{\partial}
\newcommand{\Ktwofive}{$K_{2,\downarrow 5}$}
\newcommand{\calC}{\mathcal{C}}
\newcommand{\calS}{\mathcal{S}}
\newcommand{\calT}{\mathcal{T}}
\newcommand{\calL}{\mathcal{L}}
\newcommand{\roots}[2]{#1_{\downarrow #2}}
\newcommand{\dom}{\mathrm{dom}}
\newcommand{\id}{\mathrm{id}}
\newcommand{\vampco}{\ref{thm:VH}-counterexample}
\newcommand{\denco}{\ref{thm:density}-counterexample}

\title{Every graph with no $\target$ minor is $6$-colorable}
\author{Zden\v{e}k Dvo\v{r}\'ak\thanks{Charles University, Prague, Czech Republic.  E-mail: {\tt rakdver@iuuk.mff.cuni.cz}.  Supported by ERC-CZ project LL2328 (Beyond the Four Color Theorem) of the Ministry of Education of Czech Republic.}\\
\and Sergey Norin\thanks{Department of Mathematics and Statistics, McGill University. Email: {\tt sergey.norin@mcgill.ca}. Supported by an NSERC Discovery grant.}
\and Neil Rahman \thanks{Department of Mathematics and Statistics, McGill University. Email: {\tt neil.rahman@mcgill.ca}. Partially supported by an NSERC Discovery grant and supported by an NSERC CGRS M award.}
}
\date{23 August 2026}

\begin{document}
\maketitle

\begin{abstract}
The first open case of Hadwiger's conjecture states that every $K_7$-minor-free
graph is $6$-colorable.  We prove that this is the case for $\target$-minor-free graphs,
where $\target$ denotes the graph obtained from $K_7$ by deleting two
independent edges.  The proof is based on an independently interesting density result:
Every $5$-connected $\target$-minor-free graph with $n\ge 6$ vertices has at most $4n-8$ edges.
\end{abstract}


\section{Introduction and main result}

Well-known Hadwiger's conjecture states that for every positive integer $k$, every $K_{k+1}$-minor-free graph
is $k$-colorable.  Thus, for $k=4$, it provides a strengthening of the Four Color Theorem.
In fact, Wagner's characterization of $K_5$-minor-free graphs~\cite{wagner} shows that the $k=4$ case is equivalent to the Four Color Theorem
(the cases $k\le 3$ are substantially simpler and were proved already by Hadwiger~\cite{hadwiger}).
The case $k=5$ was only resolved in 1993 by Robertson, Seymour, and Thomas~\cite{robertsonseymourthomas},
by an involved argument showing it to be also equivalent to the Four Color Theorem.

However, the next case $k=6$ is widely open.  Indeed, we do not even know whether all $K_7$-minor-free graphs are $7$-colorable,
much less $6$-colorable as postulated by Hadwiger's conjecture (Jakobsen~\cite{jakobsen1971weakenings} and later Albar and Gonçalves~\cite{eightcol} proved that they are $8$-colorable).
Furthermore, even proving weaker statements is challenging.  Let $K_k^-$, $K_k^{\vee}$, and $K_k^=$ denote the graphs obtained from the clique $K_k$ by removing a single edge,
two edges incident with the same vertex, and two independent edges, respectively.  Jakobsen~\cite{jakobsenminus} proved that all $K_7^-$-minor-free graphs are $7$-colorable,
but whether they are $6$-colorable remains open.  Moreover, Jakobsen also proved (in~\cite{jakobsen1971weakenings}) that every graph that contains neither $K_7^{\vee}$ nor $\target$
as a minor is $6$-colorable.  This result was recently improved by Norin and Totschnig~\cite{NorinTotschnig}, who proved that it suffices to exclude $K_7^{\vee}$.
However, they left open the question of whether excluding the other graph ($\target$) as a minor is by itself sufficient.
In this paper, we resolve this question in the affirmative.

\begin{theorem}\label{thm:coloring}
Every $\target$-minor-free graph is $6$-colorable.
\end{theorem}

To understand why this case is more difficult, we need to speak a bit about the proof strategy employed
in all aforementioned results.  First, one uses Kempe chain arguments to restrict the structure of a hypothetical
minimal counterexample.  More precisely, a graph $G$ is \emph{$k$-contraction-critical} if it has chromatic
number $k$ while every proper minor of $G$ is $(k-1)$-colorable.  Clearly, we can focus only on $7$-contraction-critical graphs.
Mader~\cite{Mader1968Connectivity} proved that every such graph other than $K_7$ is $7$-connected, and in particular has minimum degree at least $7$.
Moreover, it is possible to similarly show that a $7$-contraction-critical $K_7^\vee$-minor-free graph can only have very few vertices of degree exactly $7$,
and thus its average degree cannot be much smaller than $8$.  One then obtains a contradiction by using the following density result.
\begin{theorem}[Norin and Totschnig~\cite{NorinTotschnig}]\label{thm:densvee}
Let $G$ be a 4-connected graph not isomorphic to $K_{2,2,2,2}$.  If $G$ has $n\ge 5$ vertices and at least $4n-8$ edges, then
it contains $K_7^{\vee}$ as a minor.
\end{theorem}
However, Norin and Totschnig also observed that the analogue of this theorem for $\target$ is false.  Specifically, consider the graph obtained from an arbitrarily
large matching by adding four universal vertices.  This graph is $4$-connected, has average degree close to $9$, and yet does not contain $\target$
as a minor, since it is not possible to remove four vertices from $\target$ and obtain (a minor of) a matching.
Norin and Totschnig suggested that this issue could be fixed by assuming 5-connectivity,
which would be (with a bit of an extra effort) sufficient to prove Theorem~\ref{thm:coloring}.
As the main technical contribution of this paper, we provide this missing piece.
\begin{theorem}\label{thm:denstarget}
Every $5$-connected graph with $n\ge 6$ vertices and at least $4n-7$ edges contains $\target$ as a minor.
\end{theorem}
The lower bound from this theorem is close to optimal.  Note that $\target$ contains $K_6$ as a minor,
and thus any $K_6$-minor-free graph is also $\target$-minor-free.  In particular, by adding
a universal vertex to any $5$-connected plane triangulation with $n-1$ vertices, we obtain a $6$-connected $n$-vertex $\target$-minor-free graph with $4n-10$ edges.  Also, note that $K_6$ is a $\target$-minor-free graph with $n=6$ vertices and
$15=4n-9$ edges (though this seems to be an isolated example).  Thus, the following could be true.
\begin{conjecture}
Every $5$-connected graph with $n\ge 7$ vertices and at least $4n-9$ edges contains $\target$ as a minor.
\end{conjecture}
We think that our approach could in principle be used to obtain this strengthening. However, there are several places in our argument
where increasing the constant from $7$ to $9$ would require substantial additional effort, and since the exact constant
is not important for the proof of Theorem~\ref{thm:coloring}, we opted not to attempt this.
Let us also remark that there does not seem to be any fundamental obstruction that would prevent our
approach from working for $K_7^-$-minor-free graphs instead of $\target$-minor-free ones
(though of course a number of technical issues would have to be worked out).  Hence, we in particular believe the following
could be true (likely even with $2$ replaced by a slightly larger constant).
\begin{conjecture}\label{conj:k7minus}
Every $5$-connected graph with $n\ge 6$ vertices and at least $4n-2$ edges contains $K_7^-$ as a minor.
\end{conjecture}
This would be sufficient to prove that $K_7^-$-minor-free graphs are $6$-colorable.
Indeed, our proof of Theorem~\ref{thm:coloring} is by combination of Theorem~\ref{thm:denstarget}
with the following result.
\begin{theorem}\label{thm:counterdens}
Let $G$ be a $K_7^-$-minor free graph of chromatic number at least seven.  If every proper minor of $G$
is $6$-colorable, then $G$ is $7$-connected and $|E(G)|\ge 4|V(G)|-2$.
\end{theorem}
By using Conjecture~\ref{conj:k7minus} instead of Theorem~\ref{thm:denstarget}, we would get the desired
result on $6$-colorability of $K_7^-$-minor-free graphs.
This would bring us one step closer to Hadwiger's conjecture for $K_7$-minor-free graphs.
However, let us remark that a density result analogous to Theorem~\ref{thm:denstarget} is false for $K_7$-minor-free
graphs.  Indeed, graphs obtained from $5$-connected $(n-2)$-vertex plane triangulations by adding two universal vertices
are $K_7$-minor-free, $7$-connected, and have $5n-15$ edges.  Hence, making the ``final step''
from $K_7^-$-minor-free graphs to $K_7$-minor-free graphs would be substantially more difficult.

Let us now briefly describe the organization of the paper.
Section~\ref{sec:density-setup} is devoted to summarizing the definitions and previous results
used throughout the paper, and stating a refined version of Theorem~\ref{thm:denstarget}
that we need to make induction work.  Section~\ref{sec:redu} is devoted to an auxiliary result
that enables us to eliminate certain $(\le\!4)$-separations in our arguments.

The proof of Theorem~\ref{thm:denstarget} is heavily influenced by the bound on the density
of graphs without rooted $K_5$-minor, which was recently obtained by the first author~\cite{Dvorak},
and follows a similar structure.  We actually need a strengthening of one of the important
technical tools from this paper:  Each 5-rooted graph of sufficient density contains
as a minor a 7-vertex graph containing the five roots which is obtained from $K_7$ by removing
the edges between the roots and at most two additional edges forming a matching.
This is obviously useful when dealing with 5-separations in a hypothetical counterexample to Theorem~\ref{thm:denstarget},
since we then only need to show that the other part of the 5-separation can be contracted to $K_5$ on the roots
(and we apply the result of Dvořák~\cite{Dvorak} to do so).  We obtain this strengthening in Section~\ref{sec:vampire}.

Further following the approach of~\cite{Dvorak}, we actually prove a stronger version
of Theorem~\ref{thm:denstarget} which relaxes the 5-connectivity assumption to something easier to maintain:
Each $(\le\!4)$-separation has a ``light'' side, and the orientation towards the light side is consistent
in nested $(\le\!4)$-separations (see Theorem~\ref{thm:density} for an exact statement).
We then study a hypothetical minimal counterexample to this stronger theorem.
In Section~\ref{sec:five}, we apply the results of Section~\ref{sec:vampire}
to show that each 5-separation in a minimal counterexample has a ``nearly light side'' consistent
with the light or nearly light sides of nested $(\le\!5)$-separations.
The argument is finished in Section~\ref{sec:density-final}
by first arguing that each edge must be contained in many triangles (as otherwise contracting it would lead to a smaller
counterexample), and then using this to find $\target$ as a minor in a neighborhood of a small degree vertex.

The final Section~\ref{sec:coloring} is devoted to deriving our main coloring result (Theorem~\ref{thm:coloring})
from Theorem~\ref{thm:denstarget}.

\subsection{AI usage}

The main ideas of the argument (including the statements of all major results and key technical
lemmas) were provided by the co-authors.  AI was used to obtain the proofs of these results
based on an outline of the expected approach; additional guidance in the form of suggestions
for intermediate lemmas was needed several times. Several ideas present in the final paper originate from AI suggestions, namely:
\begin{itemize}
\item We originally thought that just assuming that each $(\le\!4)$-separation has a light side would be sufficient to prove the
stronger form of Theorem~\ref{thm:density}.  The AI's struggle with proving the necessary technical statements lead us to realize that
something might be wrong, and eventually we came up with the additional consistency assumption.
\item The analysis of separations after contracting part of the graph to a dart (Lemma~\ref{lemma:simplify-dense-bifragment}) was provided
by AI (and was somewhat more elegant than what we have thought of).
\item Lemma~\ref{lem:components} on neighborhoods of roots was suggested by AI and simplifies several technical issues later in the proof.
\item Lemma~\ref{lem:three-edge} was suggested by AI (but we completely replaced its proof by a simpler one).
\item A weaker form of Lemma~\ref{lem:lot-on-five} (with $K_5^=$ instead of $K_5^-$) was suggested by AI and inspired the current one
(though the proof is our own and significantly less technical than what AI proposed).
\item The idea for handling small separations in the final part of the argument (Lemma~\ref{lem:sepby4}) comes from AI.
\end{itemize}
AI also provided an initial writeup of the results, which we then heavily
edited, more than halving the  length of the paper. Finally, AI was used for
proofreading. Overall, the contribution of AI was analogous to what could be
expected from a strong Master's student. We considered listing it as a
co-author, but to match the prevailing current practice, we decided against it.

\section{Preliminaries}
\label{sec:density-setup}

Let us start by introducing definitions and auxiliary results that we need
throughout the paper.
For graphs $H$ and $G$, a \emph{model} of $H$ in $G$ is a function $\mu$ that
\begin{itemize}
\item maps vertices of $H$ to pairwise disjoint non-empty subsets of vertices of $G$ which induce connected subgraphs of $G$, and
\item maps each edge $e=uv$ of $H$ to an edge $\mu(e)$ of $G$ with one end in the set $\mu(u)$ and the other end in the set $\mu(v)$.
\end{itemize}
If there exists a model of $H$ in $G$, we say that $H$ is a \emph{minor} of $G$.
Equivalently, $H$ is a minor of $G$ if a graph isomorphic to $H$ can be obtained from a subgraph of $G$ by repeatedly contracting edges.

We are also going to consider a rooted version of this concept.  A \emph{rooted graph} is a graph $G$ together with a set $X_G$ of its vertices.
When $|X_G|=m$, we also say that the graph is \emph{$m$-rooted} to point out the number of root vertices.
For a graph $F$ without roots and a set $Z\subseteq V(F)$,
let $\roots{F}{Z}$ denote the rooted graph with the underlying graph $F$ and with the root set $Z$.

Consider a graph $H$ (without roots), a set $Y\subseteq V(H)$ of size $|X_G|$, and a bijection $\pi:Y\to X_G$.
A \emph{$\pi$-rooted model} of $H$ in $G$ is a model $\mu$ of $H$ in $G$ such that $\pi(v)\in \mu(v)$ holds for every $v\in \dom(\pi)=Y$.
If such a model exists, then we say that $H$ is a \emph{$\pi$-rooted minor} of $G$.
In other words, a graph isomorphic to $H$ via an isomorphism extending $\pi$ can be obtained from a subgraph of $G$ containing $X_G$
by repeatedly contracting edges with at least one non-root end.
We also use these definitions in the case that $H$ is a rooted graph, in which case we require that $\dom(\pi)=X_H$
(and in particular $|X_H|=|X_G|$).

When $H$ is either an unrooted graph with exactly $|X_G|$ vertices or an $|X_G|$-rooted graph,
by saying that $H$ is a \emph{rooted minor} of $G$ we mean that there exists $\pi$ such that $H$ is a $\pi$-rooted minor of $G$.
In particular, we use this convention when $H$ is a clique of size $|X_G|$, since in this case the exact bijection $\pi$ of course does not matter.

In case that $X_G\subseteq V(H)$, we usually map the roots to themselves; more precisely, we let $\id$ denote the partial
function $V(H)\to X_G$ such that $\dom(\id)=X_G$ and $\id(x)=x$ for every $x\in X_G$, and we refer to $H$ being an \emph{$\id$-rooted minor} of $G$.
Let us remark that when we apply this definition for a rooted graph $H$, we require that $X_H=X_G$.

The rooted graphs and minors naturally arise in the context of separations.
A \emph{separation} of a graph $H$ is an ordered pair $(A,B)$ of vertex sets such that
$V(H)=A\cup B$ and no edge of $H$ has one end in $A\setminus B$ and the other end in $B\setminus A$.
The \emph{order} of the separation $(A,B)$ is $|A\cap B|$, and we say that the separation is \emph{proper}
if $A\setminus B\neq \emptyset\neq B\setminus A$.  We also refer to a separation of order exactly $k$
as a \emph{$k$-separation}, and to one of order at most $k$ as an \emph{$(\le\!k)$-separation}.
Note that a graph with at least $k+1$ vertices is $k$-connected if and only if it has no proper $(\le\!(k-1))$-separation.
When $H$ is a rooted graph, we say that a separation $(A,B)$ of $H$ is a \emph{root separation} when $X_H\subseteq A$.
A rooted graph $H$ is \emph{internally $k$-connected} if it has no proper root $(\le\!(k-1))$-separation.

The \emph{right-hand side} of a separation $(A,B)$ of a (rooted or unrooted) graph $H$ is the rooted graph $R^H_{A,B}=\roots{H[B]}{(A\cap B)}$.
When the graph $H$ is clear from the context, we drop the superscript and refer to the right-hand side of the separation $(A,B)$
just by $R_{A,B}$.  Let us remark that  the right-hand side of any separation of a $k$-connected graph is internally $k$-connected.
The following fact motivates the definition of a rooted minor.
\begin{observation}\label{obs:seprepl}
Let $G$ be a graph, let $(A,B)$ be a separation of $G$, and let $H$ be a graph with $A\cap B\subseteq V(H)$ and otherwise disjoint from $A$.
If $H$ is an $\id$-rooted minor of $R^G_{A,B}$, then $G[A]\cup H$ is a minor of $G$.
Moreover, for any $Z\subseteq A$, the rooted graph $\roots{(G[A]\cup H)}{Z}$ is an $\id$-rooted minor of the rooted graph $\roots{G}{Z}$.
\end{observation}

Let $G$ be a rooted graph and let $(A,B)$ be a root separation of $G$.
An \emph{isolator} of $(A,B)$ is a root separation $(C,D)$ of minimum order subject to $C\subseteq A$ and
$B\subseteq D$.  The common order $k$ of the isolators is the \emph{root connectivity} of $(A,B)$.
The separation $(A,B)$ is \emph{linked to the roots} if its root connectivity is $|A\cap B|$.
Let us remark that when $(C,D)$ is an isolator of $(A,B)$, then $(A\cap D,B)$ is a root separation of $R_{C,D}$ whose root connectivity is exactly
$|C\cap D|$.

By Menger's theorem, the root connectivity of a root separation $(A,B)$ is equal to the maximum number of
pairwise vertex-disjoint paths from $X_G$ to $B$ in $G$.  Note that we can choose the paths
so that
\begin{itemize}
\item each of the paths intersects $X_G$ only in its first vertex,
\item each of the paths intersects $B$ only in its last vertex (and this vertex belongs to $A\cap B$), and
\item for each vertex $x\in B\cap X_G$, one of the paths consists only of $x$.
\end{itemize}
We say such a system of $k$ paths is a \emph{root linkage} of $(A,B)$, and we call the ends of the paths in $A\cap B$
the \emph{terminators} of the root linkage.
Let us note the following observation proved by contracting the paths of a root linkage.
\begin{observation}\label{obs:linkminor}
Let $G$ be a $k$-rooted graph, let $(A,B)$ be a root separation of $G$ whose root connectivity is $k$, and let $Z$ be the set of terminators of a root linkage of $(A,B)$.
Then $\roots{G[B]}{Z}$ is a rooted minor of $G$.  In particular, if $(A,B)$ is linked to the roots (that is, $k=|A\cap B|$), then $R_{A,B}$ is a rooted minor of $G$.
\end{observation}

Dvořák~\cite{Dvorak} gives an optimal sufficient condition on the existence of $K_5$ as a rooted minor in an internally 5-connected 5-rooted graph.
Note that a connectivity assumption is necessary: A 5-rooted graph can contain an arbitrarily dense component disjoint from the roots
and still be rooted $K_5$-minor-free.  However, full 5-connectivity actually is not needed (and indeed, relaxing it is necessary for
Dvořák's argument to go through, since maintaining 5-connectivity in induction would be very difficult): It suffices to assume that
each root $(\le\!4)$-separation has ``sparse'' right-hand side.  More precisely, for a rooted graph $G$,
let us define
\begin{align*}
 n(G)&=|V(G)\setminus X_G|\\
 \rho(G)&=|E(G)\setminus E(G[X_G])|\\
 \rhoFour(G)&=\rho(G)-4n(G).
\end{align*}
We say that $\rhoFour(G)$ is the \emph{$4$-density} of $G$.  Note that the $4$-density of $G$ is unaffected by the edges between roots.
A $k$-rooted graph $G$ is \emph{$4$-light} if $\rhoFour(R^G_{A,B})\le 0$ holds for every root $(\le\!(k-1))$-separation of $G$.
Note that an internally $k$-connected $k$-rooted graph is automatically $4$-light, since every root $(\le\!(k-1))$-separation $(A,B)$ of $G$
satisfies $B\subseteq A$, and thus its right-hand side has $4$-density $0$.  Moreover, as we have alluded to earlier, $4$-lightness
is significantly easier to maintain in inductive arguments than internal $5$-connectivity.

Let us note two simple observations on $4$-density.
For a rooted graph $G$ and a vertex $v\in V(G)\setminus X_G$, let $\deg^X_G v$ denote the number of neighbors of $v$ in $X_G$.
\begin{observation}\label{obs:4sum}
For every rooted graph $G$,
$$\rhoFour(G)=\frac{1}{2}\sum_{v\in V(G)\setminus X_G} (\deg_G v+\deg^X_G v-8).$$
\end{observation}
\begin{proof}
To see this, we only need to argue that $\rho(G)=\tfrac{1}{2}\sum_{v\in V(G)\setminus X_G} (\deg_G v+\deg^X_G v)$.
This is the case, since the sum counts each edge of $G$ with at least one non-root end exactly twice.
\end{proof}
Next, let us consider the effect of edge contraction on $4$-density.
Let $e=uv$ be an edge of a rooted graph $G$, where $v\not\in X_G$, and let $S$ be the set of all common neighbors of $u$ and $v$ in $G$.
We let
$$t_G(e)=\begin{cases}
|S|&\text{if $u\not\in X_G$}\\
|S\setminus X_G|+\deg^X_G v - 1 &\text{if $u\in X_G$.}
\end{cases}$$
That is, $t_G(e)$ is the number of common neighbors of $u$ and $v$ in the graph obtained from $G$ by adding all edges between its roots.

\begin{observation}\label{obs:contrdens}
Let $G$ be a rooted graph and let $e=uv$ be an edge of $G$.  If $v\not\in X_G$, then
$$\rhoFour(G/e)=\rhoFour(G)+3-t_G(e).$$
\end{observation}
\begin{proof}
We can assume that $G[X_G]$ is a clique, since the edges between roots affect neither $\rhoFour$ nor $t_G$.
Then $t_G(e)$ is exactly the number of triangles in $G$ containing $e$.  Note that contracting $e$ eliminates the
edge $e$ itself, as well as $t_G(e)$ additional edges incident with the non-root end $v$ that are identified with the edges incident with $u$.
Therefore, $\rho(G/e)=\rho(G)-1-t_G(e)$, and the claim follows since $n(G/e)=n(G)-1$.
\end{proof}

Sometimes, it is convenient to think about (root) separations in the following alternate way.
For a graph $H$, a \emph{fragment} of $H$ is a non-empty set $Y\subseteq V(H)$.  If $H$ is a rooted graph, we additionally
require that $Y\cap X_H=\emptyset$.  We define $\bd_HY$ as the set of vertices in $V(H)\setminus Y$ with a neighbor in $Y$,
$\rho(H,Y)$ as the number of edges of $H$ with at least one end in $Y$, and
$\rhoFour(H,Y)=\rho(H,Y)-4|Y|$.  Then $(V(H)\setminus Y,Y\cup \bd_HY)$ is a (root) separation of $H$ of order $|\bd_HY|$
whose right-hand side has $4$-density $\rhoFour(H,Y)$.
We also define $R^H_Y$ (or just $R_Y$ when $H$ is clear from the context)
as the right-hand side of this separation, that is, the rooted graph $\roots{H[Y\cup \bd_HY]}{\bd_HY}$.
For a non-negative integer $k$, if $|\bd_HY|=k$, then $Y$ is a \emph{$k$-fragment}, and if $|\bd_HY|\le k$, then
we say that $Y$ is an \emph{$(\le\!k)$-fragment}.  For notational convenience, we also define $\bd_H\emptyset=\emptyset$ and $\rhoFour(H,\emptyset)=0$.
Finally, let us note the following alternate characterization of $4$-light rooted graphs.
\begin{observation}\label{obs:heavyfrag}
For a positive integer $k$, a $k$-rooted graph $H$ is $4$-light if and only if every $(\le\!(k-1))$-fragment $Y$ of $H$ satisfies $\rhoFour(H,Y)\le 0$.
\end{observation}

Dvořák~\cite{Dvorak} actually provides the density result not just for $K_5$, but also for all spanning subgraphs of $K_5$.
For a class $\calC$ of graphs with $k$ vertices, we say that a $k$-rooted graph $G$ is \emph{$\calC$-universal} if for
every $F\in\calC$ and every bijection $\pi:V(F)\to X_G$, the graph $F$ is a $\pi$-rooted minor of $G$.
For a positive integer $k$ and for $t\in\bigl\{0,\ldots,\binom{k}{2}\bigr\}$, let $\calS_{k,t}$ denote the class of all graphs on $k$
vertices with at most $t$ edges.  Let $\calS_{5,4}^{-}$ be $\calS_{5,4}$ with the graphs isomorphic to
$K_2+K_3$ (the vertex-disjoint union of $K_2$ and a triangle) omitted.  For an integer $t$, we define the \emph{$t$-target} $\calT_t$ as
\[
\begin{array}{c|cccccccc}
t& t\leq0&1&2&3&4&5&6&t\geq7\\
\hline
\calT_t&\calS_{5,0}&\calS_{5,1}&
\calS_{5,3}&\calS_{5,4}^{-}&\calS_{5,6}&
\calS_{5,8}&\calS_{5,9}&\calS_{5,10}.
\end{array}
\]
We are now ready to state the main result of~\cite{Dvorak}.

\begin{theorem}[{Dvořák\cite[Theorem 4]{Dvorak}}]\label{thm:D-target}
Every $4$-light $5$-rooted graph $G$ is $\calT_{\rhoFour(G)}$-universal.
\end{theorem}

We will also need a similar density result for graphs with at most four roots
(which easily follows from the characterization by Fabila-Monroy and Wood~\cite{fabila2013rooted}).

\begin{theorem}[{Dvořák~\cite[Corollary 13]{Dvorak}}]\label{thm:D-le4}
Let $G$ be a $4$-light $k$-rooted graph, where $k\le 4$.
If $\rhoFour(G)>0$, then either
\begin{itemize}
\item $G$ contains $K_k$ as a rooted minor, or
\item $k=4$, $E(G[X_G])=\emptyset$, $\rhoFour(G)=1$, and $G$ is $\calS_{4,5}$-universal.
\end{itemize}
\end{theorem}

Concerning our main density result, Theorem~\ref{thm:denstarget},
we would like to relax the 5-connectivity to something similar to 4-lightness.  However, this is complicated by the
unrooted setting, and it turns out that the obvious first attempt (requiring every $(\le\!4)$-separation
to have a side of non-positive 4-density) is not sufficient for the proof to go through.
Hence, we need the following more restrictive assumption, essentially stating that the choice of the ``sparse''
side must be consistent for nested $(\le\!4)$-separations.

A \emph{bifragment}
in a graph $G$ is a pair $(S,T)$ of disjoint fragments such that $G$ has no edge with one end in $S$ and the other end in $T$.
It is a \emph{$k$-bifragment} if $|\bd_GS|, |\bd_GT|=k$, and \emph{$(\le\!k)$-bifragment} if $|\bd_GS|, |\bd_GT|\le k$.
We say that the bifragment
is \emph{dense} if $\rhoFour(G,S)>0$ and $\rhoFour(G,T)>0$.  We say that a graph $G$ is \emph{$4$-bilight} if it does not have any dense $(\le\!4)$-bifragment.
Thus, a stronger version of our main result can be stated as follows.

\begin{theorem}\label{thm:density}
Let $G$ be a graph with $n\ge 3$ vertices.  If $G$ is $4$-bilight and has at least $4n-7$ edges,
then it contains $\target$ as a minor.
\end{theorem}

Since a 5-connected graph cannot contain any $(\le\!4)$-bifragment, it is trivially $4$-bilight;
hence Theorem~\ref{thm:density} clearly implies Theorem~\ref{thm:denstarget}.
It will be occasionally useful to apply the $\rhoFour$ function to unrooted graphs as well; in the case that $G$ is an unrooted graph,
let us define $\rhoFour(G)=|E(G)|-4|V(G)|$.  Thus, the assumption of Theorem~\ref{thm:density} can also be written as $\rhoFour(G)\ge -7$.

To deal with 5-separations in our proof of Theorem~\ref{thm:density}, we need an auxiliary result on 5-rooted graphs.
A $5$-rooted graph $G$ is \emph{quite heavy} if either $\rhoFour(G)\ge 2$, or $\rhoFour(G)=1$ and no vertex in $V(G)\setminus X_G$ is adjacent to all five roots
(in particular, this excludes the case that $G$ has only one non-root vertex).
A \emph{vampire} is a $7$-vertex $5$-rooted graph $W$ with no edges between roots whose non-root vertices $p$ and $q$ are adjacent
and there exist distinct roots $x_1,x_2\in X_W$ such that $p$ is adjacent to every root except
possibly to $x_2$, and $q$ is adjacent to every root except possibly to $x_1$.
Let \Ktwofive{} denote the $7$-vertex $5$-rooted graph whose non-root vertices $p$ and $q$ do not form an edge and are adjacent to all roots.
Thus, both a vampire and \Ktwofive{} only miss edges of a matching among those incident with $p$ and $q$.
Dvořák~\cite{Dvorak} (essentially) proved that if a $4$-light $5$-rooted graph is quite heavy, then it is either
$\{K_4+K_1\}$-universal, or contains a vampire or \Ktwofive{} as a rooted minor
(this is basically Corollary~16 from~\cite{Dvorak}, though the exact statement is a bit different).
We need a stronger version of this auxiliary result: It is actually possible to remove the first outcome.

\begin{theorem}\label{thm:VH}
Let $G$ be a $4$-light $5$-rooted graph.  If $G$ is quite heavy, then
$G$ contains a vampire or \Ktwofive{} as a rooted minor.
\end{theorem}

The proof of this Theorem is somewhat technical and essentially involves following the outline of the proof of
Theorem~\ref{thm:D-target} in this simplified situation; we present it in Section~\ref{sec:vampire}.

\section{Reducible $(\le\!4)$-fragments}\label{sec:redu}

As the first step in the proofs of both Theorem~\ref{thm:density} and Theorem~\ref{thm:VH},
we eliminate certain (root) $(\le\!4)$-separations.  This section is devoted to auxiliary results needed for this task.
The basic instance of this reduction is as follows:  Suppose that in a minimal counterexample $G$ to one of the Theorems,
we have a (root) $(\le\!4)$-separation $(A,B)$ whose right-hand side contains $K_{|A\cap B|}$ as a rooted minor.
We would naturally like to contract this right-hand side to the clique, thus obtaining a smaller counterexample $H$ and
a contradiction.

An issue is that this reduction can break $4$-lightness or $4$-bilightness.  Indeed, $H$ can have a
(root) $(\le\!4)$-separation $(C,D)$ with $A\cap B\subseteq D$, and because of the added edges of the clique on $A\cap B$, we
can have $\rhoFour(R^H_{C,D})>\rhoFour(R^G_{C,D\cup B})$, making $\rhoFour(R^H_{C,D})>0$ possible.  To counteract this, we would like to argue that in this case
we can further contract the right-hand side of $(C,D)$ in $H$ (or equivalently, the right-hand side of $(C,D\cup B)$ in $G$).
However, it is not quite clear that contracting $R^H_{C,D}$ to a clique is possible, especially in the case that $|C\cap D|>|A\cap B|$ since then the clique on $A\cap B$
does not help us much with this task.

To deal with this issue, we essentially want to apply Theorem~\ref{thm:D-le4} to $R^H_{C,D}$.  However, the second outcome where we do not obtain
a full clique is problematic.  Fortunately, it turns out that we can make do with another rooted minor.  A \emph{dart} is a $4$-rooted graph $D$ with one
non-root vertex $y$ such that $y$ is adjacent to all roots and $D[X_D]$ is the disjoint union of a path
on three vertices and an isolated vertex.  First, let us prove a variation on Theorem~\ref{thm:D-le4}, showing that positive $4$-density ensures the existence of a dart minor.
\begin{lemma}\label{lem:dartexists}
Let $G$ be a $4$-light $4$-rooted graph.
If $\rhoFour(G)>0$, then $G$ contains a dart as a rooted minor.
\end{lemma}
\begin{proof}
We prove the claim by induction on $|V(G)|$.
Without loss of generality, we can assume that $X_G$ is an independent set, since deleting edges between roots affects neither $4$-lightness of $G$ nor the value of $\rhoFour(G)$.

Suppose first that there exists a partition $\{Z_1,Z_2\}$ of $X_G$ into sets of size two such that $G$ does not contain two vertex-disjoint paths from $Z_1$ to $Z_2$.
By Menger's theorem, there exists a separation $(B_1,B_2)$ of $G$ of order at most one such that $Z_1\subseteq B_1$ and $Z_2\subseteq B_2$.
For $i\in\{1,2\}$, let $H_i=\roots{G[B_i]}{(Z_i\cup (B_1\cap B_2))}$, and let $q$ be the number of edges of $G[X_G\cup (B_1\cap B_2)]$.
Since $X_G$ is an independent set in $G$, we have $q=0$ if $B_1\cap B_2\subseteq X_G$ and $q\le 4$ if $B_1\cap B_2$ consists of a single non-root vertex.
Observe that
$$\rhoFour(G)=\rhoFour(H_1)+\rhoFour(H_2)+q-4|B_1\cap B_2\setminus X_G|\le \rhoFour(H_1)+\rhoFour(H_2).$$
Since $G$ is $4$-light and $|X_{H_1}|,|X_{H_2}|\le 3$, observe furthermore that $\rhoFour(H_1),\rhoFour(H_2)\le 0$,
and thus $\rhoFour(G)\le 0$.  This is a contradiction.  Therefore, for every partition $\{Z_1,Z_2\}$ of $X_G$ into sets of size two,
there exist two vertex-disjoint paths from $Z_1$ to $Z_2$ in $G$.

Let $C_1$, \ldots, $C_m$ be the vertex sets of the components of $G-X_G$.  Note
that
$$\sum_{i=1}^m\rhoFour(G,C_i)=\rhoFour(G)>0,$$
and thus there exists $i\in\{1,\ldots,m\}$ such that $\rhoFour(G,C_i)>0$.
Since $G$ is $4$-light, this implies that $|\bd_GC_i|\ge 4$, and thus $\bd_GC_i=X_G$.
Consequently, $G[C_i\cup X_G]$ contains a tree $T$ whose leaves are exactly the four roots.
Note that $T$ contains a unique path $P$ between its vertices of degree at least three
(the path $P$ can also consist of just a single vertex of degree four in $T$).
Let us choose the tree~$T$ so that the path $P$ is as short as possible, and let $v_1$ and $v_2$ be its
ends (where $v_1=v_2$ is possible).  Let $X_G=\{x_1,x_2,y_1,y_2\}$, where $T$ consists
of $P$ and of pairwise edge-disjoint paths $X_1$ from $x_1$ to $v_1$, $Y_1$ from $y_1$ to $v_1$, $X_2$ from $x_2$ to $v_2$, and $Y_2$ from $y_2$ to $v_2$.

Suppose now that $G$ has a separation $(A_1,A_2)$ such that $A_1\cap A_2=\{v_1,v_2\}$ and $|X_G\cap A_1\setminus A_2|=|X_G\cap A_2\setminus A_1|=2$.
For $i\in \{1,2\}$, let $G_i$ be the $4$-rooted graph $\roots{G[A_i]}{(X_G\cap A_i)\cup \{v_1,v_2\}}$, and let $p$ be the number of edges of $G[X_G\cup \{v_1,v_2\}]$ (all incident with $v_1$ or $v_2$).
Note that $p\le 9$.  If $p=9$, then letting $\mu(x_1)=\{x_1\}$, $\mu(x_2)=\{x_2\}$, $\mu(y_1)=\{y_1\}$, $\mu(y)=\{v_1\}$, and $\mu(y_2)=\{y_2,v_2\}$ gives us
an $\id$-rooted model of a dart with vertex set $X_G\cup \{y\}$ in $G$.  Hence, suppose that $p\le 8$.
Observe that
$$\rhoFour(G)=\rhoFour(G_1)+\rhoFour(G_2)+p-8\le \rhoFour(G_1)+\rhoFour(G_2),$$
and since $\rhoFour(G)>0$, we can by symmetry assume that $\rhoFour(G_1)>0$.
As we have argued before, $G$ contains two vertex-disjoint paths from $X_G\cap A_1$ to $X_G\cap A_2$; let $P_1$ and $P_2$ be their segments from $\{v_1,v_2\}$ to $X_G\cap A_2$.
Observe that every root separation of $G_1$ corresponds to a root separation of $G$ with the same right-hand side, and since $G$ is 4-light, it follows that $G_1$ is 4-light as well.
By the induction hypothesis, $G_1$ contains a dart as a rooted minor.  By combining the corresponding model with the paths $P_1$ and $P_2$,
we see that a dart is also a rooted minor of $G$.

Therefore, we can assume that $G$ has no such separation $(A_1,A_2)$.  It follows that $G-\{v_1,v_2\}$ contains a path $Q$ from $\{x_1,y_1\}$ to $\{x_2,y_2\}$.
Let $Q_0$ be a minimal segment of this path between $X_1\cup Y_1$ and $X_2\cup Y_2$.  From the minimality of $P$,
we see that $Q_0$ is disjoint from $P$.  Hence, we can without loss of generality assume that $Q_0$ has one end on the path $X_1-v_1$ and the other end on the path $X_2-v_2$;
let $q$ denote the end of $Q_0$ on $X_2-v_2$.

Similarly, $G-\{v_1,v_2\}$ contains a path $R$ from $\{x_1,x_2\}$ to $\{y_1,y_2\}$.
Let $R_0$ be a minimal segment of $R$ between $X_1\cup X_2\cup Q_0$ and $Y_1\cup Y_2$.
From the minimality of $P$, we see that $R_0$ is disjoint from $P$.
Hence, we can without loss of generality assume that $R_0$ has one end on $(X_1-v_1)\cup (Q_0-q)$.
Let $r$ denote the other end of $R_0$.
Let us define $\mu(x_1)=V(X_1\cup Q_0\cup R_0)\setminus \{v_1,q,r\}$, $\mu(x_2)=V(X_2-v_2)$, $\mu(y_1)=V(Y_1-v_1)$, $\mu(y_2)=V(Y_2-v_2)$, and $\mu(y)=V(P)$.
Then $\mu$ is an $\id$-rooted model of a dart with vertex set $X_G\cup \{y\}$ in $G$.
\end{proof}

Thus, our aim will actually be to eliminate (root) $(\le\!4)$-separations whose right-hand side can be contracted
to a clique or to a dart.  More precisely, a $(\le\!4)$-fragment $Y$ in a rooted or unrooted graph $G$ is \emph{reducible} if
\begin{itemize}
\item[(i)] $\rhoFour(G,Y)\le 0$,
\item[(ii)] if $G$ is an unrooted graph then $|V(G)\setminus Y|\ge 3$, and
\item[(iii)] either
\begin{itemize}
\item[(a)] $R^G_Y$ contains $K_{|\bd_GY|}$ as a rooted minor, or
\item[(b)] $|Y|\ge 2$, $|\bd_GY|=4$, and $R^G_Y$ contains a dart as a rooted minor.
\end{itemize}
\end{itemize}
In the case (a), let $D$ be the rooted clique with vertex set and the root set both equal to $\bd_GY$, and in the case (b) let $D$ be a dart with $X_D=\bd_GY$ and otherwise disjoint from $G$
chosen so that $D$ is an $\id$-rooted minor of $R^G_Y$.  The \emph{$Y$-reducent} of $G$ is defined as the (rooted) graph $(G-Y)\cup D$; thus, the $Y$-reducent is an ($\id$-rooted) minor of $G$.
Let us remark that $\bd_GY=\emptyset$ is allowed and in that case the $Y$-reducent is equal to $G-Y$.

Let us comment a bit on the assumptions of reducibility.  We have already discussed the rationale behind the choice of the minors in (iii);
the assumption on $|Y|$ means that the reduction actually decreases the number of vertices of the graph.
The condition (i) ensures that performing the reduction does not decrease the 4-density of the whole graph, and thus
e.g. in the context of Theorem~\ref{thm:density} the reducent still satisfies the $\rhoFour\ge -7$ assumption.
The condition (ii) is needed to prevent us from completely erasing the whole graph (and, more specifically,
to preserve the $n\ge 3$ assumption of Theorem~\ref{thm:density}).

Thus, it remains to argue that the reduction does not violate the $4$-bilightness assumption of Theorem~\ref{thm:density}
and the $4$-lightness assumption of Theorem~\ref{thm:VH}.  The argument is essentially the same for both Theorems
and we give it in the rest of the section.  To avoid repetitions, it will be convenient to extend the notation around bifragments
to rooted graphs.  A \emph{$(\le\!4)$-bifragment} in a rooted graph $G$ is a pair $(S,\emptyset)$, where $S$ is a $(\le\!4)$-fragment. 
Let us recall that in a rooted graph, fragments cannot contain roots.
The $(\le\!4)$-bifragment $(S,\emptyset)$ is \emph{dense} if $\rhoFour(G,S)>0$; and
a rooted graph $G$ is $4$-bilight if it does not have any dense $(\le\!4)$-bifragment.
Note that with this definition, a $5$-rooted graph is $4$-bilight if and only if it is $4$-light.
Let us start with a technical lemma.
\begin{lemma}\label{lemma:simplify-dense-bifragment}
Let $H$ be an unrooted or $5$-rooted graph and let $U$ be a set of vertices of $H$ such that either
\begin{itemize}
\item[(i)] $U$ induces a clique in $H$, or
\item[(ii)] $|U|=5$, there exists a vertex $y\in U$ of degree four adjacent in $H$ exactly to the vertices in $U\setminus \{y\}$,
and $H[U\setminus\{y\}]$ contains a 3-vertex path as a subgraph.
\end{itemize}
If $H$ is not $4$-bilight, then it contains a dense $(\le\!4)$-bifragment $(S,T)$ such that either
\begin{itemize}
\item $U\cap (S\cup T)=\emptyset$, or
\item $U\cap S\neq\emptyset$, $|U\setminus(S\cup\bd_HS)|\le 1$, $U\cap T=\emptyset$, and the rooted graph $R^H_S$ is 4-light.
\end{itemize}
\end{lemma}
\begin{proof}
Choose a dense $(\le\!4)$-bifragment $(S,T)$ in $H$ so that $|S|+|T|$ is smallest possible,
and so that $|U\cap S|\ge |U\cap T|$.  Note that the second condition is trivially true when $H$ is a rooted graph.
The minimality of $|S|+|T|$ ensures the following property ($\dag$):
\begin{center}
For each $Z\in\{S,T\}$, if $|U\cap Z|=1$, then $U\subseteq Z\cup \bd_HZ$.
\end{center}
Indeed, suppose for a contradiction that say $U\cap S$ consists of a single vertex~$z$, and that there also exists a vertex $u\in U\setminus (S\cup \bd_HS)$.
Then $uz\not\in E(H[U])$, and thus we are in the case (ii) and $u,z\in U\setminus\{y\}$.  Since both $u$ and $z$ are adjacent to $y$, it follows that $y\in \bd_H S$.
Let $S'=S\setminus \{z\}$.  Note that $\bd_HS'\subseteq (\bd_HS\cup \{z\})\setminus \{y\}$, and thus $|\bd_HS'|\le |\bd_HS|$.  Moreover, let $m\le 4$ denote the number of
neighbors of $z$ in $\bd_HS$ and note that $\rhoFour(H,S')=\rhoFour(H,S)+4-m\ge \rhoFour(H,S)$.  Therefore, $(S',T)$ is a dense $(\le\!4)$-bifragment of $H$, and since $|S'|<|S|$,
this contradicts the choice of $(S,T)$.

If $U\cap S=\emptyset$, then since $|U\cap S|\ge |U\cap T|$, we also have $U\cap T=\emptyset$, and thus the first outcome holds. 
Hence, we can assume that $|U\cap S|\ge 1$.

This implies that $|U\setminus (S\cup\bd_HS)|\le 1$.  Indeed, if $U\not\subseteq S\cup\bd_HS$, then
($\dag$) implies $|U\cap S|\neq 1$, and thus $|U\cap S|\ge 2$.
Moreover, the vertices in $U\cap S$ are non-adjacent to those in $U\setminus (S\cup\bd_HS)$, and thus we are in the case (ii), $y\in \bd_HS$, and $|(U\setminus \{y\})\cap S|\ge 2$.
Since $|U\setminus\{y\}|=4$ and $H[U\setminus\{y\}]$ contains a 3-vertex path, this implies that at least three vertices of $U\setminus\{y\}$ belong to or are adjacent to a vertex of $(U\setminus \{y\})\cap S$
and thus $|(U\setminus \{y\})\cap (S\cup\bd_HS)|\ge 3$ and $|U\setminus (S\cup\bd_HS)|=|(U\setminus \{y\})\setminus (S\cup\bd_HS)|\le 1$.

In particular, we have $|U\cap T|\le 1$. Moreover, note that $U\cap S\neq \emptyset$ implies $U\not\subseteq T\cup \bd_HT$,
and thus by ($\dag$) we have $|U\cap T|\neq 1$.  It follows that $U\cap T=\emptyset$.

Finally, note that the minimality of $|S|+|T|$ ensures that the rooted graph $R^H_S$ is 4-light, as otherwise we could by Observation~\ref{obs:heavyfrag}
replace $S$ by a proper subfragment of $S$.  Therefore, the second outcome of the lemma holds.
\end{proof}

We can now give the main technical result of this section, showing that the reduction of a (maximal) reducible $(\le\!4)$-fragment preserves $4$-bilightness.

\begin{lemma}\label{lemma:preserve-bilight}
Let $G$ be an unrooted or $5$-rooted $4$-bilight graph, let $Y$ be an inclusionwise-maximal reducible $(\le\!4)$-fragment of $G$, and let $H$ be the $Y$-reducent of $G$.
Then $H$ is $4$-bilight and $\rhoFour(H)\ge\rhoFour(G)$.  Moreover, if $E(H)\setminus E(G)$ contains an edge between a non-root vertex of degree at least five
and a root vertex, then $\rhoFour(H)>\rhoFour(G)$.
\end{lemma}
\begin{proof}
Let $D$ be as in the definition of the $Y$-reducent and in case that $D$ is a dart, let $y$ denote the unique non-root vertex of $D$.
Let $p$ be the number of edges in $E(H)\setminus E(G)$ that join a non-root vertex of degree at least five and a root vertex.
When $D$ is a dart, then such edges are not incident with $y$, since $y\not\in X_H$ has degree four.
Note that $\rhoFour(D)=0$ and that
$$\rhoFour(H)\ge \rhoFour(G)-\rhoFour(G,Y)+\rhoFour(D)+p\ge \rhoFour(G)+p,$$
since reducibility of $Y$ implies $\rhoFour(G,Y)\le 0$.

Therefore, it suffices to argue that $H$ is $4$-bilight.  Suppose for a contradiction that this is not the
case, and thus $H$ contains a dense $(\le\!4)$-bifragment $(S,T)$ satisfying one of the two outcomes of Lemma~\ref{lemma:simplify-dense-bifragment}
with $U=V(D)$.  The outcome $V(D)\cap (S\cup T)=\emptyset$ is not possible, since then $(S,T)$ would also be a dense $(\le\!4)$-bifragment in $G$.
Hence, we have $V(D)\cap S\neq\emptyset$, $|V(D)\setminus(S\cup\bd_HS)|\le 1$, $V(D)\cap T=\emptyset$, and the rooted graph $R^H_S$ is 4-light.

If $|\bd_HS|\le 3$, then let $M$ be the rooted clique with $|\bd_HS|$ vertices and with $X_M=V(M)$, and if $|\bd_HS|=4$, then let $M$ be a dart.
Since $R^H_S$ is 4-light and $\rhoFour(R^H_S)=\rhoFour(H,S)>0$,
Theorem~\ref{thm:D-le4} in the former case and Lemma~\ref{lem:dartexists} in the latter case imply that there exists a bijection $\pi:X_M\to \bd_HS$
such that $M$ has a $\pi$-rooted model $\mu$ in $R^H_S$.

Let us now extend the fragment $S$ slightly, while still ensuring the existence of a rooted minor related to $M$.  Specifically, let $S_1$ and $M_1$ be defined as follows:
\begin{itemize}
\item If $|V(D)\setminus (S\cup\bd_HS)|=1$, then let $v$ be the unique vertex in $V(D)\setminus (S\cup\bd_HS)$, and note that $D$ is a dart and $y\in \bd_HS$.
Let $S_1=S\cup\{y\}$ and note that $\bd_HS_1=(\bd_HS\setminus\{y\})\cup \{v\}$.  Thus, $|\bd_HS_1|=|\bd_H S|$, and since $T\cap (\bd_HS\cup V(D))=\emptyset$, we have $(S_1\cup \bd_HS_1)\cap T=\emptyset$.
Moreover, $M_1=M$ is a rooted minor of $R^H_{S_1}$, since its rooted model can be obtained from $\mu$ by adding $v$ to the bag of $\mu$ containing $y$.
\item If $V(D)\subseteq S\cup\bd_HS$, $D$ is a dart, and $y\in \bd_HS$, then let $S_1=S\cup \{y\}$ and note that $\bd_HS_1=\bd_HS\setminus\{y\}$.  Hence $|\bd_HS_1|<|\bd_H S|$ and
$(S_1\cup \bd_HS_1)\cap T=\emptyset$.  Let $M_1$ be the rooted clique with $V(M_1)=X_{M_1}=X_M\setminus \{\pi^{-1}(y)\}$ and let $\pi_1$ be the restriction of $\pi$ to $X_{M_1}$,
and observe that $M_1$ is a $\pi_1$-rooted minor of $R^H_{S_1}$:  This is clear when $M$ is a clique.  When $M$ is a dart with $V(M)\setminus X_M=\{y'\}$ and with $M[X_M]$ containing the path $abc$ and isolated vertex $d$,
then we can define a $\pi_1$-rooted model $\mu_1$ of $M_1$ in $R^H_{S_1}$ as follows:
\begin{itemize}
\item If $y=\pi(a)$, then $\mu_1(b)=\mu(b)$, $\mu_1(c)=\mu(c)$, and $\mu_1(d)=\mu(d)\cup \mu(y')$.
\item If $y=\pi(b)$, then $\mu_1(a)=\mu(a)$, $\mu_1(c)=\mu(c)\cup \mu(b)$, and $\mu_1(d)=\mu(d)\cup \mu(y')$.
\item If $y=\pi(c)$, then $\mu_1(a)=\mu(a)$, $\mu_1(b)=\mu(b)$, and $\mu_1(d)=\mu(d)\cup \mu(y')$.
\item If $y=\pi(d)$, then $\mu_1(a)=\mu(a)$, $\mu_1(b)=\mu(b)$, and $\mu_1(c)=\mu(c)\cup \mu(y')$.
\end{itemize}
\item Otherwise, we have $V(D)\subseteq S\cup\bd_HS$ and $V(D)\setminus X_D\subseteq S$; we let $S_1=S$ and $M_1=M$.
\end{itemize}
In either case $(S_1,T)$ is a $(\le\!4)$-bifragment in $H$ such that $S\subseteq S_1$, $V(D)\subseteq S_1\cup\bd_HS_1$, $V(D)\setminus X_D\subseteq S_1$,
and $M_1$ is a rooted minor of $R^H_{S_1}$.  Note that $|S|\ge 2$, since $|\bd_HS|\le 4$ and $\rhoFour(H,S)>0$, and thus $|S_1|\ge 2$.

Let us now undo the reduction of $Y$.  More precisely, let $Y'=(S_1\setminus (V(D)\setminus X_D))\cup Y$.  Since $X_D=\bd_G Y$, we have $\bd_GY'=\bd_HS_1$. Moreover, $|Y'|\ge |S_1|-1+|Y|\ge |Y|+1\ge 2$.
It follows that $Y'$ is a $(\le\!4)$-fragment of $G$ such that $Y\subsetneq Y'$.
Since $M_1$ is a rooted minor of $R^H_{S_1}$, Observation~\ref{obs:seprepl}
implies that $M_1$ is also a rooted minor of $R^G_{Y'}$.  Moreover, when $G$ is an unrooted graph, we have $|V(G)\setminus Y'|\ge |T\cup\bd_G T|>3$,
since $\rhoFour(G,T)=\rhoFour(H,T)>0$.

Since $Y\subsetneq Y'$ and $Y$ is an inclusionwise-maximal reducible $(\le\!4)$-fragment of $G$,
it follows that the $(\le\!4)$-fragment $Y'$ is not reducible.  We have argued that it satisfies the conditions (ii) and (iii)
from the definition of reducibility, and thus (i) has to be false, that is, $\rhoFour(G,Y')>0$.
However, then $(Y',T)$ is a dense $(\le\!4)$-bifragment in $G$.  This is a contradiction, since $G$ is $4$-bilight.
\end{proof}

\section{Vampires and \Ktwofive{}'s}
\label{sec:vampire}

In this section, we aim to prove Theorem~\ref{thm:VH}.  For convenience, let us record a consequence of its outcome.
\begin{observation}\label{obs:vatodart}
Let $G$ be a $5$-rooted graph and let $Z$ be a proper subset of $X_G$.  If $G$ contains a vampire or \Ktwofive{}
as a rooted minor, then $n(G)\ge 2$ and $\roots{(G-(X_G\setminus Z))}{Z}$ contains either $K_{|Z|}$ or a dart (when $|Z|=4$) as a rooted minor.
\end{observation}
\begin{proof}
Let $W$ be a vampire or \Ktwofive{} contained in $G$ as a rooted minor.
Clearly $n(G)\ge n(W)=2$. Note that $W[X_G]$ has no edges, and thus $W'=\roots{(W-(X_G\setminus Z))}{Z}$ is a rooted minor of $\roots{(G-(X_G\setminus Z))}{Z}$.
Hence, it suffices to show that either $K_{|Z|}$ or a dart (when $|Z|=4$) is a rooted minor of $W'$.
This is trivial if $|Z|\le 1$, and thus suppose that $|Z|\in \{2,3,4\}$.
Let $V(W')\setminus Z=\{v_1,v_2\}$, where $\deg_{W'} v_1\ge \deg_{W'} v_2$.
Let us now distinguish cases based on $|Z|$.
\begin{itemize}
\item If $|Z|=2$, then note that either $W'$ is connected (when $W$ is a vampire) or contains a vertex adjacent to all vertices of $Z$
(when $W$ is \Ktwofive{}).  In either case, we can contract a path between the two vertices of $Z$ through $W'$ to a single edge
and obtain $K_2$ as a rooted minor of $W'$.
\item Suppose next that $|Z|=3$.  Note that each vertex of $V(W')\setminus Z$ is adjacent to at least two vertices of $Z$.
In particular, let $z_2$ and $z_3$ be two neighbors of $v_2$ in $Z$, and let $z_1$ be the remaining vertex in $Z\setminus\{z_2,z_3\}$.
If $v_1$ is adjacent to all vertices of $Z$, then we can contract the edges $v_1z_1$ and $v_2z_2$ and obtain $K_3$ as a rooted minor
of $W'$.  Otherwise $W$ is a vampire, $v_1v_2\in E(W')$, and $v_1$ is adjacent to $z_1$ and one of $z_2$ and $z_3$.
By symmetry, we can assume that $v_1$ is adjacent to $z_3$.  Then we can again contract the edges $v_1z_1$ and $v_2z_2$
and obtain $K_3$ as a rooted minor of $W'$.
\item Finally, suppose that $|Z|=4$.  The vertex $v_2$ has at least three neighbors in $Z$, which we denote by $z_2$, $z_3$, and $z_4$.
The vertex $v_1$ is adjacent to at least two of them, say $z_3$ and $z_4$, as well as the remaining vertex $z_1$ in $Z\setminus\{z_2,z_3,z_4\}$.
Moreover, $v_1$ is adjacent to $v_2$ if $W$ is a vampire and to $z_2$ if $W$ is \Ktwofive{}.
Hence, by contracting the edge $z_2v_2$ of $W'$, we obtain a dart as a rooted minor.
\end{itemize}
\end{proof}

\begin{corollary}\label{cor:vatoredu}
Let $G$ be a $(\le\!5)$-rooted $4$-light graph and let $(A,B)$ be a root $5$-separation of $G$.
If $|X_G|\le 4$, then additionally assume that $\rhoFour(G)\le 0$.
If $R_{A,B}$ contains a vampire or \Ktwofive{} as a rooted minor, then either
\begin{itemize}
\item[(i)] $|X_G|=5$ and $G$ contains a vampire or \Ktwofive{} as a rooted minor, or
\item[(ii)] $G$ contains a reducible $(\le\!4)$-fragment $Y$ such that $B\setminus A\subseteq Y$.
\end{itemize}
\end{corollary}
\begin{proof}
Let $W$ be a vampire or \Ktwofive{} which appears in $R_{A,B}$ as a rooted minor.
If the root separation $(A,B)$ is linked to the roots (and in particular $|X_G|=|A\cap B|=5$),
then Observation~\ref{obs:linkminor} implies that $W$ is also a rooted minor of $G$ and the outcome (i) holds.

Therefore, we can assume that this is not the case, and thus the root connectivity $k$ of $(A,B)$ is at most four.
We also have $k\le |X_G|$.  If $k=|X_G|$, then let $(C,D)=(X_G,V(G))$ (which is an isolator of $(A,B)$),
otherwise let $(C,D)$ be an arbitrary isolator of $(A,B)$.  Note that $\rhoFour(R_{C,D})\le 0$:
If $k<|X_G|$, this follows from the assumption that $G$ is $4$-light, and if $k=|X_G|$,
then $R_{C,D}=G$ and $\rhoFour(G)\le 0$ by the assumptions.

Let $A'=A\cap D$ and note that $(A',B)$ is a root separation of $R_{C,D}$ of root connectivity $k=|C\cap D|$.
Let $Z$ be the set of terminators of a root linkage of the root separation $(A',B)$ in $R_{C,D}$.
By Observation~\ref{obs:vatodart}, $\roots{G[B]}{Z}$ contains as a rooted minor $K_{|Z|}$ or a dart;
and by Observation~\ref{obs:linkminor}, so does $R_{C,D}$.
Let $Y=D\setminus C$.  Note that $R_Y=R_{C,D}$, and thus $\rhoFour(G,Y)\le 0$.
Moreover, $|Y|\ge |B\setminus A|\ge 2$, since $W$ is a rooted minor of $R_{A,B}$.
It follows that the $(\le\!4)$-fragment $Y$ is reducible, and the outcome (ii) holds.
\end{proof}

The proof of Theorem~\ref{thm:VH} proceeds by contradiction.
A \emph{\vampco{}} is a $4$-light $5$-rooted graph which is quite heavy,
but contains neither a vampire nor \Ktwofive{} as a rooted minor.
A \vampco{} $G$ is \emph{minimal} if every \vampco{} either has more than $|V(G)|$ vertices,
or exactly $|V(G)|$ vertices and at least $|E(G)|$ edges.
Note that a minimal \vampco{} $G$ satisfies $E(G[X_G])=\emptyset$, since edges between roots affect neither the assumptions
nor the outcome of Theorem~\ref{thm:VH}.
As the first step in the argument, we analyze small root separations.

\begin{lemma}\label{lem:vh-no-reducible}
A minimal \vampco{} has no reducible $(\le\!4)$-fragment.
\end{lemma}
\begin{proof}
Suppose for a contradiction that a minimal \vampco{} $G$ contains a reducible $(\le\!4)$-fragment.
Let $Y$ be an inclusionwise-maximal reducible $(\le\!4)$-fragment in $G$
and let $H$ be the $Y$-reducent of $G$.  Recall that in $5$-rooted graphs, $4$-bilightness coincides with $4$-lightness.
Hence, Lemma~\ref{lemma:preserve-bilight} implies that $H$ is $4$-light, $\rhoFour(H)\ge\rhoFour(G)$, and
if $E(H)\setminus E(G)$ contains an edge between a non-root vertex of degree at least five
and a root vertex, then $\rhoFour(H)\ge \rhoFour(G)+1$.

We claim that $H$ is quite heavy.  Since $G$ is quite heavy and $\rhoFour(H)\ge\rhoFour(G)$, this is the case
unless $\rhoFour(G)=1$ and there exists a non-root vertex $v$ adjacent to all five roots in $H$ (but not in $G$).
However, then $E(H)\setminus E(G)$ contains an edge between the vertex $v$ of degree at least $|X_H|=5$ and a root,
and thus we actually have $\rhoFour(H)\ge \rhoFour(G)+1\ge 2$; and $H$ is quite heavy in this case as well.

Since $|V(H)|<|V(G)|$, the 5-rooted graph $H$ is not a \vampco{}, and thus it contains a vampire or \Ktwofive{} as a rooted minor.
However, then $G$ contains such a rooted minor as well, since $H$ is an $\id$-rooted minor of $G$.  This is a contradiction, since $G$ is a \vampco{}.
\end{proof}

\begin{lemma}\label{lem:order-five}
A minimal \vampco{} has no proper root $5$-separation with quite heavy right-hand side.
\end{lemma}
\begin{proof}
Suppose for a contradiction that a minimal \vampco{} $G$ has a proper root $5$-separation $(A,B)$ with a quite heavy right-hand side.
Since $G$ is $4$-light, the right-hand side $R_{A,B}$ of this root separation is also $4$-light.
Since the root separation $(A,B)$ is proper, we have $|B|<|V(G)|$, and thus $R_{A,B}$ is not a \vampco{};
hence, $R_{A,B}$ contains a vampire or \Ktwofive{} as a rooted minor.
We now obtain a contradiction by applying Corollary~\ref{cor:vatoredu}:  The first outcome is not
possible, since $G$ is a \vampco{}, and the second one is excluded by Lemma~\ref{lem:vh-no-reducible}.
\end{proof}

\noindent Next, let us investigate the neighborhood of roots.

\begin{lemma}\label{lem:components}
If $G$ is a minimal \vampco{}, then the graph $G-X_G$ is connected,
each non-root vertex is adjacent to at most three roots,
and each root has degree at least two.
\end{lemma}
\begin{proof}
Let $C_1$, \ldots, $C_k$ be the vertex sets of the components of $G-X_G$,
where $\rhoFour(G, C_1)\ge\ldots\ge\rhoFour(G,C_k)$.
Since $G$ is $4$-light, for each $i\in\{1,\ldots,k\}$ such that $\rhoFour(G,C_i)>0$, we have $|\bd_GC_i|>4$,
and thus $\bd_GC_i=X_G$.  Hence, if $k\ge 2$ and $\rhoFour(G,C_2)>0$, then we can contract each of $C_1$ and $C_2$
to a single vertex and obtain \Ktwofive{} as a rooted minor of $G$.  This is not possible since $G$ is a \vampco{},
and thus $\rhoFour(G,C_i)\le 0$ for every $i\in\{2,\ldots,k\}$.
However, then
$$\rhoFour(G)=\sum_{i=1}^k \rhoFour(G,C_i)\le\rhoFour(G,C_1),$$
and thus the subgraph $R^G_{C_1}$ of $G$ is quite heavy.
By Lemma~\ref{lem:order-five}, the root $5$-separation $(V(G)\setminus C_1,X_G\cup C_1)$ of $G$ is not proper.
This implies that $V(G)\setminus C_1=X_G$, and thus $k=1$.  That is, the graph $G-X_G=G[C_1]$ is connected.
Moreover, note that since $\bd_GC_1=X_G$, each root of $G$ has degree at least one.

Let us now consider a non-root vertex $z$, and suppose that $z$ is adjacent to $d\ge 4$ roots.
Let $x\in X_G$ be a root such that all vertices of $X_G\setminus\{x\}$ are adjacent to $z$.
Let $Q_1$, \ldots, $Q_m$ be the vertex sets of the components of $G-(X_G\cup \{z\})$;
since the graph $G-X_G$ is connected, we have $z\in \bd_GQ_i$ for each $i\in\{1,\ldots,m\}$.
Moreover,
$$\rhoFour(G)=d-4+\sum_{i=1}^m\rhoFour(G,Q_i).$$
Note that $d\le 5$.  Moreover, since $G$ is quite heavy, if $\rhoFour(G)<2$, then $\rhoFour(G)=1$ and $d=4$ (since no vertex is adjacent to
all roots in this case).  Therefore, we have
$$\sum_{i=1}^m\rhoFour(G,Q_i)\ge \rhoFour(G)+4-d\ge 1.$$
Without loss of generality, we can assume that
$\rhoFour(G,Q_1)>0$.  Since $G$ is $4$-light, we have $|\bd_GQ_1|\ge 5$,
and thus $|X_G\cap \bd_GQ_1|\ge 4$.  If $z$ is adjacent to $x$ or if $x\in \bd_GQ_1$,
then $z$ together with a vertex obtained by contracting $Q_1$ would give a vampire as a rooted minor of $G$,
a contradiction.  Therefore, $xz\not\in E(G)$ and $\bd_GQ_1=\{z\}\cup (X_G\setminus\{x\})$.
Since each root vertex has degree at least one in $G$, we can assume that $m\ge 2$ and $x$ has a neighbor in $Q_2$.
But then we obtain a vampire as a rooted minor of $G$ by contracting each of $Q_1$ and $Q_2\cup\{z\}$ to a single vertex.
This is a contradiction, and thus $\deg^X_G z\le 3$ for every $z\in V(G)\setminus X_G$.

We have already argued that each root of $G$ has at least one neighbor.
Suppose that a root $y$ of $G$ has exactly one neighbor $v$.
Recall that $G[X_G]=\emptyset$, since $G$ is a minimal \vampco{}, and thus $v\in V(G)\setminus X_G$.
Consider the root $5$-separation $(C,D)=(X_G\cup\{v\},V(G)\setminus \{y\})$ of $G$.
Note that
$$\rhoFour(R_{C,D})=\rhoFour(G)+4-\deg^X_G v\ge \rhoFour(G)+1\ge 2,$$
and thus $R_{C,D}$ is quite heavy.
This contradicts Lemma~\ref{lem:order-five}.  Therefore, each root of $G$ has degree at least two.
\end{proof}

With this, we can prove that contracting and deleting edges of a minimal \vampco{}
leads to a decrease in density.

\begin{lemma}\label{lem:edge-light}
If $G$ is a minimal \vampco{}, then for every edge $e\in E(G)$, neither of the 5-rooted graphs $G-e$ and $G/e$ is quite heavy.
\end{lemma}
\begin{proof}
Recall that $E(G[X_G])=\emptyset$, and thus the edge $e=uv$ has a non-root end, say $v$.
Consider a 5-rooted graph $H\in\{G-e,G/e\}$, and suppose for a contradiction that $H$ is quite heavy.
Since $H$ is a rooted minor of $G$, it contains neither a vampire nor \Ktwofive{} as a rooted minor.
By the minimality of $G$, the 5-rooted graph $H$ is not a counterexample,
and thus it is not 4-light.  Let $(C,D)$ be a root $(\le\!4)$-separation of $H$ such that $\rhoFour(R^H_{C,D})\ge 1$,
chosen so that $D$ is inclusionwise-minimal among all such root $(\le\!4)$-separations.

Let us first consider the case that $H=G-e$.  If $e\in E(G[C])$ or $e\in E(G[D])$, then $(C,D)$ would also be a $(\le\!4)$-separation
of $G$ and $\rhoFour(R^G_{C,D})\ge\rhoFour(R^H_{C,D})>0$, contradicting the 4-lightness of $G$.
Therefore, we can assume that $u\in C\setminus D$ and $v\in D\setminus C$.
Let $D_1=D\cup\{u\}$, so that $(C,D_1)$ is a root separation of $G$ of order $|C\cap D|+1\le 5$
and $R^G_{C,D_1}$ is obtained from $R^H_{C,D}$ by adding the root vertex $u$ and the edge $e$.
Hence, we have
$$\rhoFour(R^G_{C,D_1})=\rhoFour(R^H_{C,D})+1\ge 2.$$
Since $G$ is 4-light, it follows that the root separation $(C,D_1)$ of $G$ has order five and is quite heavy.
By Lemma~\ref{lem:order-five}, this root separation cannot be proper, and since $v\in D_1\setminus C$,
it follows that $C=X_G$.  But then $v$ is the only neighbor of the root $u$ in $G$, 
which contradicts Lemma~\ref{lem:components}.

Finally, let us consider the case that $H=G/e$.  By a slight abuse of notation, we use $u$ to refer to the vertex of $H$ resulting from the contraction of the edge $e$ (as well as to the vertex of $G$).
Note that $u\in D$, as otherwise $(C\cup \{v\},D)$ would be a root separation of $G$ with the right-hand side equal to $R^H_{C,D}$,
contradicting the 4-lightness of $G$.

Suppose next that $u\in D\setminus C$, and let $D_1=D\cup\{v\}$, so that $(C,D_1)$ is a root separation of $G$ of order $|C\cap D|$.
The inclusionwise-minimality of $D$ from the choice of the root separation $(C,D)$ implies that
the rooted graph $R^H_{C,D}$ is $4$-light.
By Theorem~\ref{thm:D-le4} and Lemma~\ref{lem:dartexists}, $R^H_{C,D}$ contains either $K_{|C\cap D|}$ or a dart as a rooted minor.
The rooted graph $R^H_{C,D}$ is a rooted minor of $R^G_{C,D_1}$ obtained by contracting the edge $e$, and thus
$R^G_{C,D_1}$ contains either $K_{|C\cap D|}$ or a dart as a rooted minor as well.  Moreover, $|D_1 \setminus C|\ge 2$, since $u,v\in D_1\setminus C$.
Since $G$ is 4-light, the $(\le\!4)$-fragment $D_1\setminus C$ of $G$ is reducible, contradicting Lemma~\ref{lem:vh-no-reducible}.

Therefore, we have $u\in C\cap D$.  Let $C_2=C\cup\{v\}$ and $D_2=D\cup\{v\}$,
so that $(C_2,D_2)$ is a root separation of $G$ of order $|C\cap D|+1\le 5$.
Clearly $C_2\neq X_G$, since $v\not\in X_G$.
Let $t$ be the number of common neighbors of $u$ and $v$ in $D_2\setminus C_2$, and observe that
$$\rhoFour(R^G_{C_2,D_2})=t+\rhoFour(R^H_{C,D})\ge t+1>0.$$
Since $G$ is $4$-light, the root separation $(C_2,D_2)$ has order five.
Moreover, if a non-root vertex of $R^G_{C_2,D_2}$ is adjacent to all its roots,
then $t\ge 1$ and $\rhoFour(R^G_{C_2,D_2})\ge 2$.  Therefore, the $5$-rooted graph $R^G_{C_2,D_2}$ is quite heavy.
This contradicts Lemma~\ref{lem:order-five}.
\end{proof}

This implies that minimal \vampco{}s are not too dense.
\begin{corollary}\label{cor:vhsparse}
Every minimal \vampco{} $G$ satisfies $\rhoFour(G)=1$.
\end{corollary}
\begin{proof}
Suppose for a contradiction that $\rhoFour(G)\ge 2$, and let $e$ be any edge of $G$.
Lemma~\ref{lem:edge-light} implies that the 5-rooted graph $G-e$ is not quite heavy.
Since $G$ is quite heavy and $\rhoFour(G-e)=\rhoFour(G)-1$, it follows that $\rhoFour(G)=2$, $\rhoFour(G-e)=1$,
and a vertex $u$ of $G-e$ is adjacent to all roots.   However, then $u$ is adjacent to all roots in $G$ as well,
and this contradicts Lemma~\ref{lem:components}.
\end{proof}

On the other hand, edges of a minimal \vampco{} must lie in many triangles.

\begin{lemma}\label{lem:triangles}
Let $G$ be a minimal \vampco{}, let $e$ be any edge of $G$, let $H=G/e$
and let $u$ be the vertex of $H$ obtained by contracting $e$.
Then $t_G(e)\ge 3$, and if $t_G(e)=3$, then $u\not\in X_H$
(or equivalently, $e$ is not incident with a root of $G$) and $u$ is adjacent in $H$ to all roots.
\end{lemma}
\begin{proof}
By Observation~\ref{obs:contrdens}, we have
$$\rhoFour(G/e)=\rhoFour(G)+3-t_G(e)=4-t_G(e).$$
Lemma~\ref{lem:edge-light} implies that the 5-rooted graph $G/e$ is not quite heavy,
and thus $t_G(e)\ge 3$.

Moreover, if $t_G(e)=3$, then $\rhoFour(G/e)=1$, and thus $G/e$ has a non-root vertex $z$ adjacent to
all roots.  If $z\neq u$, then $z$ would also be adjacent in $G$ to all (at least four) roots
not incident with $e$, contradicting Lemma~\ref{lem:components}.  Therefore, we have $z=u$.
\end{proof}

Let us improve this bound a bit in a special case.

\begin{lemma}\label{lem:three-edge}
Let $G$ be a minimal \vampco{} and let $e=uv$ be an edge of $G$ incident with a non-root vertex $v$.
If $\deg_G v+\deg^X v\leq8$, then $t_G(e)\geq4$.
\end{lemma}
\begin{proof}
Suppose for a contradiction that $t_G(e)\le 3$.  By a slight abuse of notation, let us also use $u$ to denote the vertex of $H=G/e$ obtained by contracting the edge
$e$.  By Lemma~\ref{lem:triangles}, we have $t_G(e)=3$, $u\not\in X_G$, and $u$ is adjacent in $H$ to all roots.
Equivalently, each vertex in $X_G$ has a neighbor in $\{u,v\}$.
Recall that Lemma~\ref{lem:components} gives $\deg^X_G u,\deg^X_G v\le 3$, and thus $\max(\deg^X_G u,\deg^X_G v)=3$
and $\min(\deg^X_G u,\deg^X_G v)\ge 2$.

Let $Q_1$, \ldots, $Q_m$ be the vertex sets of the components of $G-(X_G\cup \{u,v\})$.
Note that since $G-X_G$ is connected by Lemma~\ref{lem:components}, for each $i\in \{1,\ldots,m\}$, there is an edge of $G$ between $Q_i$ and $\{u,v\}$.
If there existed $i\in \{1,\ldots,m\}$ such that at least four vertices of $X_G$ have a neighbor in $Q_i$, then
we could obtain a vampire as a rooted minor of $G$ by contracting $Q_i$ to a single vertex and by contracting the edge $uv$.
Hence, this is not the case, and in particular $|\bd_G Q_i|\le 5$ for every $i\in\{1,\ldots, m\}$.

Let $I\subseteq\{1,\ldots,m\}$ consist of the indices $i\in\{1,\ldots, m\}$ such that $\rhoFour(G,Q_i)>0$.
Since $G$ is $4$-light, for each $i\in I$ we have $|\bd_G Q_i|=5$, and thus $u,v\in \bd_G Q_i$.
Hence, $$\deg_G v\ge\deg^X_G v+|I|+1.$$
Moreover, Lemma~\ref{lem:order-five} implies that for each $i\in I$ we have $\rhoFour(G,Q_i)=1$ and there exists a vertex $q_i\in Q_i$
adjacent to all vertices of $\bd_GQ_i$.
Note that
\begin{align}
1&=\rhoFour(G)=\deg^X u+\deg^X v-7+\sum_{i=1}^m \rhoFour(G,Q_i)\nonumber\\
&\le \deg^X u+\deg^X v-7+|I|=\min(\deg^X_G u,\deg^X_G v)-4+|I|\nonumber\\
&\le \deg^X_G v+|I|-4\le \deg_G v-5,\label{eq:vdeg}
\end{align}
and thus $\deg_G v\ge 6$.  Recall that $\min(\deg^X_G u,\deg^X_G v)\ge 2$, and
since $\deg_G v+\deg^X v\leq8$, it follows that $\deg^X_G v=2$ and $\deg_G v=6$.
Thus, the inequalities in \eqref{eq:vdeg} all hold with equality, and in particular $|I|=3$.
Therefore, the vertex $v$ is adjacent to two vertices of $X_G$, the vertex $u$, the vertex $q_i$ for each $i\in I$,
and no other vertices.  Let $W$ be the set of the two neighbors of $v$ in $X_G$.

Consider any $i\in I$.  Note that the only common neighbors of $v$ and $q_i$ are $u$ and possibly the vertices of $W$,
and thus $t_G(vq_i)\le 3$.  By Lemma~\ref{lem:triangles}, we have $t_G(vq_i)\ge 3$, and thus
$q_i$ is adjacent to both vertices in $W$.  Moreover, Lemma~\ref{lem:triangles} implies that since $t_G(vq_i)=3$,
each root has a neighbor in $\{v,q_i\}$, and thus also the vertices in $X_G\setminus W$ are adjacent to $q_i$.
However, this implies that $\{u,v\}\cup X_G\subseteq \bd_GQ_i$, which is a contradiction.
\end{proof}

\noindent Next we prove a technical lemma on small graphs with large minimum degree.
\begin{lemma}\label{lem:vampnbhd}
Let $H$ be a graph with $n$ vertices and let $X$ be a proper subset of $V(H)$.
Suppose that $n+|X|\le 8$, each vertex $v\in V(H)\setminus X$
has degree at least four, and each vertex $v\in X$ has degree at least $5-|X|$.
For every set $Z\subseteq V(H)$ of size five such that $X\subset Z$, there exists
a vertex $z\in Z\setminus X$ such that $\roots{H}{Z}$ has an $\id$-rooted minor containing
at least three edges between $z$ and $Z\setminus\{z\}$.
\end{lemma}
\begin{proof}
Note that $n\ge 5$, since vertices in $V(H)\setminus X$ have degree at least four,
and thus $n+|X|\le 8$ implies $|X|\le 3$.  Let $U=V(H)\setminus Z$.  We can assume that each vertex $v\in Z\setminus X$
has at most two neighbors in $Z$, as otherwise we can simply choose $z=v$.
It follows that at least two neighbors of $v$ are in $U$,
and thus $|U|\ge 2$, $n=|Z|+|U|\ge 7$, and $|X|\le 8-n\le 1$.  It follows that $H$ has minimum degree at least
$\min(4,5-|X|)=4$.

If a vertex $u\in U$ has at least four neighbors in $Z$, then we obtain the desired $\id$-rooted minor
by contracting the edge $uz$ for any neighbor $z$ of $u$ in $Z\setminus X$.  Therefore, we can assume that this
is not the case, and thus each vertex in $U$ has at most three neighbors in $Z$
and at least one neighbor in $U$.  Since $|U|=n-5\le 3$, it follows
that the subgraph $H[U]$ is connected.

Recall that all (at least four) vertices in $Z\setminus X$ have a neighbor (actually at least two neighbors) in $U$.
Hence, we can choose any vertex $z\in Z\setminus X$ and obtain the desired $\id$-rooted minor by
contracting the connected subgraph $H[U\cup\{z\}]$.
\end{proof}

\noindent We are now ready to finish the proof.

\begin{proof}[Proof of Theorem~\ref{thm:VH}]
Suppose for a contradiction that Theorem~\ref{thm:VH} is false, i.e., there exists a \vampco{}, and let $G$ be a minimal one.
By Corollary~\ref{cor:vhsparse}, we have $\rhoFour(G)=1$, and thus Observation~\ref{obs:4sum}
gives
\begin{equation}\label{eq:degsum}
\sum_{v\in V(G)\setminus X_G} (\deg_G v+\deg^X_G v-8)=2.
\end{equation}
Note that for every $v\in V(G)\setminus X_G$, we have $\deg^X_G v\le 3$ by Lemma~\ref{lem:components}.
Moreover $\deg_G v\le \deg^X_G v+n(G)-1$, and thus
$$\deg_G v+\deg^X_G v-8\le 2\deg^X_G v+n(G)-9\le n(G)-3.$$
It follows that $n(G)\cdot (n(G)-3)\ge 2$, and thus $n(G)\ge 4$.

By~\eqref{eq:degsum}, there exists a vertex $v\in V(G)\setminus X_G$ such that 
$\deg_G v+\deg^X_G v\le 8$.  Let $H$ be the subgraph of $G$ induced by the neighbors of $v$
and let $X=X_G\cap V(H)$.  Note that $X\neq V(H)$, since the graph $G-X_G$ is connected by Lemma~\ref{lem:components}
and $n(G)>1$.  Observe that for each vertex $u\in V(H)\setminus X$, we have $\deg_H u=t_G(uv)$,
and for each $u\in X$, we have $\deg_H u=t_G(uv)-(|X|-1)$.
Lemma~\ref{lem:three-edge} implies that $\deg_H u\ge 4$ for each $u\in V(H)\setminus X$ and $\deg_H u\ge 5-|X|$
for each $u\in X$.  In particular, we have $|V(H)|\ge 5$.

Consider the root separation $(A,B)=(V(G)\setminus\{v\},V(H)\cup\{v\})$ of $G$.
Let $(C,D)$ be an isolator of $(A,B)$, where $(C,D)$ is chosen as $(X_G,V(G))$ when the root connectivity of $(A,B)$ is five.
Let $Y$ be the set of terminators of a root linkage $\calL$ of the root separation $(A\cap D, B)$ of $R_{C,D}$. Then $X \subseteq Y$.
Let $Z\subseteq V(H)$ be any superset of $Y$ of size five.
By Lemma~\ref{lem:vampnbhd}, there exists a vertex $z\in Z\setminus X$ and a graph $F_0$ with vertex set $Z$ such that
$\deg_{F_0} z\ge 3$ and $\roots{H}{Z}$ contains $F_0$ as an $\id$-rooted minor.  Let $F$ be the graph obtained from $F_0$ by adding
a vertex $v'$ adjacent to all vertices of $Z$; then clearly $\roots{G[B]}{Z}$ contains $F$ as an $\id$-rooted minor.
Note that $z$ has at most one non-neighbor in $F$.

If the root connectivity of $(A,B)$ is five, then recall that $(C,D)=(X_G,V(G))$, and thus
$z\not\in C$.  In the rooted graph $G$, let us contract $\roots{G[B]}{Z}$ to $F$, then contract each path of the linkage $\calL$ to a single vertex,
except for the path that ends in $z$ which is contracted to an edge instead.  In this way, we obtain a vampire
as a rooted minor of $G$, which is a contradiction.

Therefore, the root connectivity of $(A,B)$ (which is equal to $|C\cap D|$ and $|Y|$) is at most four.
Let $F'$ be the $\id$-rooted minor of $\roots{G[B]}{Y}$ obtained as follows.  We first contract
$\roots{G[B]}{Z}$ to $F$.  If $z\in Y$, we let $F'=F-(Z \setminus Y)$ and $z'=z$.
If $z\not\in Y$ and $|Y|\ge 3$, then let $z'$ be a neighbor of $z$ in $Y$ and let $F'$ be the graph obtained
from $F-(Z\setminus (Y \cup \{z\}))$ by contracting the edge $zz'$.  Finally, if $z\not\in Y$ and $|Y|\le 2$,
then let $F'=F-(Z\setminus Y)$ and choose $z'\in Y$ arbitrarily.  In all cases, note that
$F'$ is a graph with vertex set $Y\cup \{v'\}$, where $v'$ is a universal vertex and $z'$ has at most one non-neighbor in $F'$.
Let $w'$ be this non-neighbor (or an arbitrary vertex in $Y\setminus \{z'\}$ if $|Y|\ge 2$ and $z'$ is adjacent to all
vertices of $Y\setminus \{z'\}$, or $w'=z'$ if $|Y|=1$).

In the rooted graph $R^G_{C,D}$ we contract $\roots{G[B]}{Y}$ to $F'$,
in the case that $|C\cap D|=|Y|\le 3$ we additionally contract the edge $v'w'$,
and then we contract the paths of the root linkage $\calL$ to single vertices.
If $|C\cap D|\le 3$, then this results in $K_{|C\cap D|}$ as a rooted minor of $R^G_{C,D}$.
If $|C\cap D|=4$, then this instead results in a dart as a rooted minor of $R^G_{C,D}$.
Moreover, note that $|D\setminus C|\ge |V(H)\cup\{v\}|-|C\cap D|\ge 2$.
Since $G$ is 4-light, it follows that the fragment $D\setminus C$ of $G$ is reducible.
However, this contradicts Lemma~\ref{lem:vh-no-reducible}.

This contradiction shows that there exists no \vampco{}.
\end{proof}

\section{Restricting the separations}
\label{sec:five}

We are now ready to start our work on the proof of Theorem~\ref{thm:density}.
The overall structure of the argument is similar to the proof of Theorem~\ref{thm:VH},
with a few adjustments for the unrooted setting.  A \emph{\denco{}} is a $4$-bilight $\target$-minor-free graph $G$
with $n\ge 3$ vertices and at least $4n-7$ edges.  By our definition of $\rhoFour$ for unrooted graphs, we equivalently have
$\rhoFour(G)\ge -7$.  A \denco{} $G$ is \emph{minimal} if every \denco{} either has more than $|V(G)|$ vertices,
or exactly $|V(G)|$ vertices and at least $|E(G)|$ edges.
We start the argument by constraining separations of small order.
\begin{lemma}\label{lem:den-no-reducible}
A minimal \denco{} has no reducible $(\le\!4)$-fragment.
\end{lemma}
\begin{proof}
Suppose for a contradiction that a minimal \denco{} $G$ contains a reducible $(\le\!4)$-fragment,
and let $Y$ be an inclusionwise-maximal one.  Let $H$ be the $Y$-reducent of $G$.  Note that since $Y$ is a reducible
fragment in an unrooted graph, we have $|V(H)|\ge |V(G)\setminus Y|\ge 3$. Since $H$ is a minor of $G$, it is $\target$-minor-free.
Moreover, Lemma~\ref{lemma:preserve-bilight} implies that $H$ is $4$-bilight and $\rhoFour(H)\ge\rhoFour(G)\ge -7$,
and thus $H$ is a \denco{}.  This contradicts the minimality of the \denco{} $G$.
\end{proof}

\noindent Let us note a technical observation.
\begin{observation}\label{obs:decrw}
Let $G$ and $H$ be $(\le\!4)$-rooted graphs with the same underlying unrooted graph such that $X_H\subseteq X_G$.
If $G$ is 4-light and $\rhoFour(G)\le 0$, then $H$ is also $4$-light and $\rhoFour(H)\le 0$.
\end{observation}
\begin{proof}
Let $M=X_G\setminus X_H$.
Let $Y$ be any fragment of $H$ such that either $Y=V(H)\setminus X_H$ or $|\bd_HY|<|X_H|$; it suffices to prove that
$\rhoFour(H,Y)\le 0$.  Let $Y'=Y\setminus M$ and note that $\bd_GY'\subseteq \bd_HY\cup M$.  In particular,
we either have $Y'=V(G)\setminus X_G$, or $|\bd_GY'|\le |\bd_HY|+|M|<|X_H|+|M|=|X_G|$.  By the assumptions,
we have $\rhoFour(G,Y')\le 0$.  Let $m$ be the number of edges with one end in $M\cap Y$ and the other end
outside of $Y'$ (this includes the edges between the vertices of $M\cap Y$).  Note that
\begin{align*}
m&\le \binom{|M\cap Y|}{2}+|M\cap Y|\cdot |\bd_H Y|=|M\cap Y|\cdot \Bigl(\frac{|M\cap Y|-1}{2}+|\bd_H Y|\Bigr)\\
&\le |M\cap Y|\cdot (|M|+|\bd_H Y|)\le |M\cap Y|\cdot |X_G|\le 4|M\cap Y|.
\end{align*}
Therefore,
$$\rhoFour(H,Y)=\rhoFour(G,Y')+m-4|M\cap Y|\le\rhoFour(G,Y')\le 0,$$
as required.
\end{proof}

We can now combine Lemma~\ref{lem:den-no-reducible} with the results of the previous section to get the following consequence.
\begin{corollary}\label{cor:no-VH-in-light}
Let $G$ be a minimal \denco{} and let $(C,D)$ be a $(\le\!4)$-separation of $G$ such that $|(C\setminus D)\cup \bd_G (C\setminus D)|\ge 3$,
$\rhoFour(R_{C,D})\le 0$ and $R_{C,D}$ is $4$-light.  Let  $(A,B)$ be a $5$-separation of $G$ such that $B \subseteq D$.
If the $5$-rooted graph $R_{A,B}$ is $4$-light, then it is not quite heavy.
\end{corollary}
\begin{proof}
Let $M$ be the set of vertices of $C\cap D$ with no neighbor in $C\setminus D$, and let $C'=C\setminus M$.
Then $(C',D)$ is also a $(\le\!4)$-separation of $G$.  Let $H=R_{C',D}$.  Observation~\ref{obs:decrw} implies that $\rhoFour(H)\le 0$
and $H$ is $4$-light.

Suppose for a contradiction that $R^G_{A,B}$ is quite heavy.  Since $R^G_{A,B}$ is also $4$-light, Theorem~\ref{thm:VH} implies
that $R^G_{A,B}$ contains a vampire or \Ktwofive{} as a rooted minor.  Let $A'=A\cap D$. Since each vertex of $X_H=C'\cap D$ has a neighbor in $C\setminus D\subseteq A\setminus B$,
we have $X_H\subseteq A$, and thus $(A', B)$ is a root $5$-separation of $H$.
Moreover, note that $R^H_{A',B}=R^G_{A,B}$. By Corollary~\ref{cor:vatoredu}
applied to $H$ and $(A',B)$, whose outcome (i) is excluded since $|X_H|=|C'\cap D|\le 4$,
we see that $H$ contains a reducible $(\le\!4)$-fragment $Y$.  Since
$|V(G)\setminus Y|\ge |C'|=|(C\setminus D)\cup\bd_G(C\setminus D)|\ge 3$,
the fragment $Y$ is also reducible in the unrooted graph $G$.  This contradicts Lemma~\ref{lem:den-no-reducible}.
\end{proof}

This implies the following consistency result on $(\le\!4)$-separations and $5$-separations.

\begin{lemma}\label{lem:one-VH-light}
Let $G$ be a minimal \denco{} and let $(A,B)$ be a $5$-separation of $G$.  If the rooted graph $R_{A,B}$ is quite heavy,
then the ``opposite'' 5-rooted graph $R_{B,A}$ is $4$-light.
\end{lemma}
\begin{proof}
Suppose for a contradiction that $R_{B,A}$ is not $4$-light, and thus there exists a $(\le\!4)$-fragment $L\subseteq A\setminus B$
such that $\rhoFour(R_{B,A}, L)>0$. Note that $\rhoFour(G,L)=\rhoFour(R_{B,A},L)$, and thus also $\rhoFour(G,L)>0$.

It follows that the $5$-rooted graph $R_{A,B}$ is $4$-light.  Indeed, otherwise there would exist a $(\le\!4)$-fragment
$U\subseteq B\setminus A$ such that $\rhoFour(G,U)=\rhoFour(R_{A,B},U)>0$.  However, then the $(\le\!4)$-bifragment $(L,U)$ of $G$ would be dense,
a contradiction since $G$ is $4$-bilight.  

Consider the $(\le\!4)$-separation $(C,D)=(L\cup\bd_G L, V(G)\setminus L)$ of $G$.
We have $|(C\setminus D)\cup\bd_G(C\setminus D)|=|L\cup\bd_G L|\ge 6$, since $\rhoFour(G,L)>0$.
Note that $\rhoFour(R_{C,D})\le 0$ and the rooted graph $R_{C,D}$ is $4$-light, as otherwise there would exist a $(\le\!4)$-fragment $U'\subseteq D\setminus C$ such that
$\rhoFour(G,U')>0$, and the dense $(\le\!4)$-bifragment $(L,U')$ would contradict the $4$-bilightness of $G$.
Since $B \subseteq D$, the 5-separation $(A,B)$ contradicts Corollary~\ref{cor:no-VH-in-light}. \end{proof}

For a later use, let us note the following consequence.
\begin{corollary}\label{cor:VH-fragment-vs-4sep}
Let $G$ be a minimal \denco{}, let $(C,D)$ be a $(\le\!4)$-separation of $G$, and let $Y$ be a 5-fragment of $G$ such that
$Y\cup\bd_GY\subseteq C$.  If $R_Y$ is quite heavy, then
\begin{itemize}
\item[(i)] $|C\setminus D|\ge 3$,
\item[(ii)] $\rhoFour(R_{C,D})\le 0$, and
\item[(iii)] $R_{C,D}$ is $4$-light.
\end{itemize}
\end{corollary}
\begin{proof}
Since $R_Y$ is quite heavy, we have $|Y|=n(R_Y)\ge 2$.  Consequently, $|C|\ge |Y|+|\bd_GY|\ge 7$,
and since $|C\cap D|\le 4$, this gives (i).

Let us now prove (ii).  Let $M$ be the set of all vertices $v\in C\cap D$ with no neighbor in $D\setminus C$
and let $D'=D\setminus M$.  Then $(C,D')$ is also a $(\le\!4)$-separation of $G$ and $\rhoFour(R_{C,D})=\rhoFour(R_{C,D'})$.
Hence, it suffices to argue that $\rhoFour(R_{C,D'})\le 0$.  Since each vertex of $C\cap D'$ has a neighbor in $D'\setminus C=D\setminus C$,
and since $\bd_GY\subseteq C$, we have $Y\subseteq C\setminus D'$.

Let $(A,B)=(V(G)\setminus Y,Y\cup\bd_GY)$.  Then $R_{A,B}=R_Y$ is quite heavy, and thus Lemma~\ref{lem:one-VH-light}
implies that $R_{B,A}$ is $4$-light.  Note that $(A\cap C,D')$ is a root $(\le\!4)$-separation of $R_{B,A}$
with the right-hand side equal to $R_{C,D'}$, and thus $\rhoFour(R_{C,D'})\le 0$.

Finally, let us prove (iii).  Consider any root $(\le\!4)$-separation $(I,J)$ of $H=R_{C,D}$ and let $I'=I\cup C$.  Then $(I',J)$ is a $(\le\!4)$-separation
of $G$ with the same right-hand side and $Y\cup \bd_GY\subseteq C\subseteq I'$, and thus (ii) gives
$\rhoFour(R^H_{I,J})=\rhoFour(R^G_{I',J})\le 0$.  It follows that $R_{C,D}$ is $4$-light.
\end{proof}

\noindent For the next step, let us show a simple consequence of Dvořák's density result.
\begin{lemma}\label{lem:D-target}
Let $G$ be a $4$-light $5$-rooted graph and let $m=10-|E(G[X_G])|$ be the number of non-edges between
its roots.  If
$$\rhoFour(G)\ge \Bigl\lceil\frac{m+3}{2}\Bigr\rceil,$$
then $K_5$ is a rooted minor of $G$.
\end{lemma}
\begin{proof}
Let $b(m)=\lceil\tfrac{m+3}{2}\rceil$.
For each integer $t$, let $f(t)$ denote the largest integer~$s \leq 10$ such that for every $t'\ge t$, the $t'$-target contains all $5$-vertex
graphs with at most $s$ edges.  The following table lists the relevant values:
\begin{center}
\begin{tabular}{c|c|c|c}
$m$      & $b(m)$ & the $b(m)$-target & $f(b(m))$\\
\hline
$0$, $1$ & $2$    & $\calS_{5,3}$     & $3$\\
$2$, $3$ & $3$    & $\calS_{5,4}^-$   & $3$\\
$4$, $5$ & $4$    & $\calS_{5,6}$     & $6$\\
$6$, $7$ & $5$    & $\calS_{5,8}$     & $8$\\
$8$, $9$ & $6$    & $\calS_{5,9}$     & $9$\\
$10$     & $7$    & $\calS_{5,10}$    & $10$
\end{tabular}
\end{center}
An inspection of this table shows that $f(b(m))\ge m$.
Since $\rhoFour(G)\ge b(m)$ by the assumptions, it follows that the $\rhoFour(G)$-target $\calT$ contains all $5$-vertex
graphs with $m$ edges.  In particular, it contains the complement $S$ of the graph $G[X_G]$.
By Theorem~\ref{thm:D-target}, it follows that $S$ is an $\id$-rooted minor of $G$.
Therefore, the graph $S\cup G[X_G]$ isomorphic to $K_5$ is an $\id$-rooted minor of $G$.
\end{proof}

We can now show that every 5-separation in a minimal \denco{} has a not quite heavy side.

\begin{lemma}\label{lem:not-both-VH}
Let $(A,B)$ be a $5$-separation of a minimal \denco{} $G$, and let $m=10-|E(G[A\cap B])|$
denote the number of missing edges between the vertices of $A\cap B$.
Then exactly one of the 5-rooted graphs $R_{A,B}$ and $R_{B,A}$ is quite heavy,
and moreover the quite heavy one has $4$-density at least $m+2$.
\end{lemma}
\begin{proof}
By symmetry, we can assume that $\rhoFour(R_{A,B})\ge \rhoFour(R_{B,A})$.
Note that
\begin{align*}
-7&\le \rhoFour(G)=\rhoFour(R_{A,B})+\rhoFour(R_{B,A})+|E(G[A\cap B])|-4|A\cap B|\\
&=\rhoFour(R_{A,B})+\rhoFour(R_{B,A})-10-m.
\end{align*}
Therefore, we have
\begin{equation}\label{eq:rhosum}
\rhoFour(R_{A,B})+\rhoFour(R_{B,A})\ge m+3,
\end{equation}
and since $\rhoFour(R_{A,B})\ge \rhoFour(R_{B,A})$ is an integer,
\begin{equation}\label{eq:rhocount}
\rhoFour(R_{A,B})\ge \Bigl\lceil\frac{m+3}{2}\Bigr\rceil\ge 2.
\end{equation}
This implies that $R_{A,B}$ is quite heavy.

Suppose now for a contradiction that $R_{B,A}$ is quite heavy as well.
By Lemma~\ref{lem:one-VH-light}, the 5-rooted graphs $R_{A,B}$ and $R_{B,A}$ are both $4$-light.
In particular, Theorem~\ref{thm:VH} implies that $R_{B,A}$ has an $\id$-rooted minor $W$,
where $W$ is a vampire or \Ktwofive{}.  On the other hand, given~\eqref{eq:rhocount},
Lemma~\ref{lem:D-target} implies that $R_{A,B}$ has an $\id$-rooted minor $K$ isomorphic to $K_5$.
Consequently, $G$ contains the graph $K\cup W$ as a minor.  However, note that $K\cup W$ is
isomorphic to $\target$ or one of its supergraphs, which is a contradiction since $G$ is $\target$-minor-free.

Therefore, the 5-rooted graph $R_{B,A}$ is not quite heavy.  It follows that $\rhoFour(R_{B,A})\le 1$, and
\eqref{eq:rhosum} gives $\rhoFour(R_{A,B})\ge m+2$.
\end{proof}

\noindent Let us restate a useful observation appearing implicitly in the proof of Lemma~\ref{lem:not-both-VH}.

\begin{lemma}\label{lem:K5-in-light-VH}
Let $(A,B)$ be a $5$-separation of a minimal \denco{} $G$.  If the 5-rooted graph $R_{A,B}$ is quite heavy and
$4$-light, then it contains $K_5$ as a rooted minor.
\end{lemma}
\begin{proof}
Let $m=10-|E(G[A\cap B])|$.  By Lemma~\ref{lem:not-both-VH}, we have
$$\rhoFour(R_{A,B})\ge m+2\ge \Bigl\lceil \frac{m+3}{2}\Bigr\rceil,$$
and thus Lemma~\ref{lem:D-target} implies that $K_5$ is a rooted minor of $R_{A,B}$.
\end{proof}

Finally, we show that the quite heavy sides are consistent between non-crossing $5$-separations.

\begin{lemma}\label{lem:no-nested-VH}
Let $G$ be a minimal \denco{} and let $(S,T)$ be a $5$-bifragment in $G$.
Then at most one of the $5$-rooted graphs $R_S$ and $R_T$ is quite heavy.
\end{lemma}
\begin{proof}
Suppose for a contradiction that $R_S$ and $R_T$ are both quite heavy.
Consider the $5$-separation $(A,B)=(V(G)\setminus S,S\cup \bd_GS)$ of $G$.
By Lemma~\ref{lem:one-VH-light}, since $R_{A,B}=R_S$ is quite heavy, the $5$-rooted graph $R_{B,A}$ is $4$-light.
Since $(S,T)$ is a bifragment, we have $T\subseteq A\setminus B$, and it follows that the $5$-rooted graph $R_T$ is $4$-light as well.
Symmetrically, the $5$-rooted graph $R_S$ is $4$-light.

We claim that $G$ contains five pairwise vertex-disjoint paths from $S\cup\bd_GS$ to $T\cup\bd_GT$.
Indeed, otherwise by Menger's theorem, there would exist a $(\le\!4)$-separation $(C,D)$ of $G$ such
that $S\cup\bd_GS\subseteq C$ and $T\cup\bd_GT\subseteq D$. 
However, then Corollary~\ref{cor:VH-fragment-vs-4sep} implies that $|C\setminus D|\ge 3$, $\rhoFour(R_{C,D})\le 0$, and $R_{C,D}$ is $4$-light;
and since $T\cup\bd_GT\subseteq D$ and $R_T$ is $4$-light and quite heavy, we get a contradiction by Corollary~\ref{cor:no-VH-in-light}.

Let $\calL$ be a set of five pairwise vertex-disjoint paths from $S\cup\bd_GS$ to $T\cup\bd_GT$ in $G$.
Without loss of generality, we can assume that these paths start in $\bd_GS$ and end in $\bd_GT$.
By Theorem~\ref{thm:VH}, we see that $R_S$
has an $\id$-rooted minor $W$, where $W$ is a vampire or \Ktwofive{}.  
On the other hand, Lemma~\ref{lem:K5-in-light-VH} shows that $R_T$ has $K_5$ as a rooted minor.
By contracting $R_S$ to $W$, $R_T$ to $K_5$ on $\bd_GT$, and contracting the edges of the paths of $\calL$,
we obtain $\target$ as a minor of $G$, which is a contradiction.
\end{proof}

Let us now consolidate several of the results of this section.
\begin{corollary}\label{cor:bifragment-constraint}
Let $G$ be a graph and let $(S,T)$ be a $(\le\!5)$-bifragment in $G$.  Suppose that for each $Z\in \{S,T\}$,
we have $\rhoFour(G,Z)>0$, and moreover if $|\bd_GZ|=5$, then the $5$-rooted graph $R_Z$ is quite heavy.
Then $G$ is not a minimal \denco{}.
\end{corollary}
\begin{proof}
Suppose for a contradiction that $G$ is a minimal \denco{}.
Note that $S$ and $T$ cannot both be $(\le\!4)$-fragments, as otherwise $(S,T)$ would be a dense $(\le\!4)$-bifragment.
Hence, we can assume that $S$ is a $5$-fragment (and $R_S$ is quite heavy).
By Lemma~\ref{lem:no-nested-VH}, we see that $T$ is a $(\le\!4)$-fragment.
However, this contradicts Lemma~\ref{lem:one-VH-light} with $(A,B)=(V(G)\setminus S,S\cup\bd_GS)$.
\end{proof}

\section{Density of $\target$-minor-free graphs}
\label{sec:density-final}

Using the results of the previous section, we now argue that removing or contracting an
edge in a minimal \denco{} preserves $4$-bilightness, and thus it has to decrease the number of edges.

\begin{lemma}\label{lemma:one-edge-delete}
If $G$ is a minimal \denco{}, then $|E(G)|=4|V(G)|-7$.
\end{lemma}
\begin{proof}
Suppose for a contradiction that $|E(G)|\ge 4|V(G)|-6$.  Since $|V(G)|\ge 3$, this gives $|E(G)|\ge 6$,
and in particular $E(G)\neq\emptyset$.  Let $e$ be any edge of $G$, and consider the graph $H=G-e$.
We have $|V(H)|=|V(G)|\ge 3$ and $|E(H)|=|E(G)|-1\ge 4|V(H)|-7$.  Moreover, $H$ is a minor of $G$,
and thus it is $\target$-minor-free.  The minimality of $G$ implies that $H$ is not a \denco{}, and thus
$H$ is not 4-bilight.

Let $(S,T)$ be a dense $(\le\!4)$-bifragment in $H$.  If $e=uv$ had both ends in $S\cup\bd_HS$, or both ends in $T\cup\bd_HT$,
or neither end in $S\cup T$, then $(S,T)$ would also be a dense $(\le\!4)$-bifragment in $G$,
a contradiction.  Hence, we can by symmetry assume that $u\in S$ and $v\not\in S\cup\bd_HS$.
Moreover, in case that $v\in T$, we also by symmetry assume that $|\bd_GS|\ge|\bd_GT|$.

Note that $|\bd_GS|=|\bd_HS|+1\le 5$.  The rooted graph $R^G_S$ is obtained from $R^H_S$ by adding the root $v$ and
the edge $e$, and thus $\rhoFour(G,S)=\rhoFour(H,S)+1\ge 2$.  Analogously, we also have $\rhoFour(G,T)\ge 2$ if $v\in T$
(and $\rhoFour(G,T)=\rhoFour(H,T)>0$ if $v\not\in T$).

Since $\rhoFour(G,S)\ge 2$, we have $|S|\ge 2$.  Let $S'=S\setminus\{u\}$.  Since $v\not\in S\cup\bd_HS$,
we have $\bd_GS'\subseteq(\bd_GS\setminus\{v\})\cup \{u\}$, and thus $|\bd_GS'|\le|\bd_GS|$.
Moreover, let $a\le 5$ denote the number of edges of $G$ between $u$ and $\bd_GS$ (including the edge $e$),
and observe that
\begin{equation}\label{eq:excle}
\rhoFour(G,S')=\rhoFour(G,S)+4-a\ge 1.
\end{equation}
If $|\bd_GS|\le 4$, then also $|\bd_G T|\le 4$ (if $v\in T$ then recall that we chose the labels of $S$ and $T$ so that $|\bd_GS|\ge|\bd_GT|$, and otherwise
$|\bd_G T|=|\bd_H T|$).  However, then $(S',T)$ would be a dense $(\le\!4)$-bifragment in $G$, a contradiction.
Therefore, $S$ is a 5-fragment in $G$.

Recall that $\rhoFour(G,T)\ge 1$, and that if $|\bd_GT|=5$,
then $v\in T$ and $\rhoFour(G,T)\ge 2$ (and thus $R^G_T$ is quite heavy).  Since $\rhoFour(G,S')\ge 1$, Corollary~\ref{cor:bifragment-constraint}
applied to the bifragment $(S',T)$ implies that $|\bd_GS'|=5$ and the 5-rooted graph $R^G_{S'}$ is not quite heavy.  Consequently,
we have $\rhoFour(G,S')=1$ and there exists a vertex $u'\in S'$ adjacent to all vertices in
$\bd_GS'=\bd_HS\cup\{u\}$.  Moreover, \eqref{eq:excle} implies that $a=5$ (that is, $u$ is adjacent to all vertices in $\bd_GS=\bd_HS\cup\{v\}$) and $\rhoFour(G,S)=2$.
By Lemma~\ref{lem:not-both-VH}, $\rhoFour(G,S)=2$ is only possible when $\bd_GS=\bd_HS\cup \{v\}$ induces a clique in $G$.
Therefore, the subgraph of $G$ induced by $\bd_HS\cup\{u,u',v\}$ contains $K_7^-$, which is a contradiction since $G$ is $\target$-minor-free.
\end{proof}

\noindent Next, we consider the effect of an edge contraction.
\begin{lemma}\label{lem:one-edge-contract}
If $G$ is a minimal \denco{}, then $|V(G)|\ge 7$ and every edge of $G$ is contained in at least four triangles.
\end{lemma}
\begin{proof}
Since $\binom{|V(G)|}{2}\ge |E(G)|=4|V(G)|-7$ and $|V(G)|\ge 3$, we have $|V(G)|\ge 7$.
Suppose for a contradiction that an edge $e=uv\in E(G)$ is contained in $t\le 3$ triangles,
and thus the graph $H=G/e$ satisfies
$$\rhoFour(H)=\rhoFour(G)+4-(t+1)\ge\rhoFour(G)=-7.$$
Moreover, $H$ is a minor of $G$, and thus it is $\target$-minor-free, and $|V(H)|=|V(G)|-1>3$.
The minimality of $G$ implies that $H$ is not a counterexample, and thus $H$ is not $4$-bilight.
Let $(S,T)$ be a dense $(\le\!4)$-bifragment in $H$, chosen so that $|S|+|T|$ is minimal.
Moreover, let $z$ denote the vertex of $H$ resulting from the contraction of the edge $e$.

Suppose first that $z\in S$.  By the minimality of $|S|+|T|$, we see that the rooted graph $R^H_S$
is $4$-light, as otherwise we could replace $S$ by a smaller $(\le\!4)$-fragment.
By Theorem~\ref{thm:D-le4} and Lemma~\ref{lem:dartexists}, the $(\le\!4)$-rooted graph
$R^H_S$ contains either $K_{|\bd_HS|}$ or a dart as a rooted minor.
Let $S'=(S\setminus\{z\})\cup\{u,v\}$.  Since $R^H_S$ is obtained from $R^G_{S'}$ by contracting the edge $e$,
the $(\le\!4)$-rooted graph $R^G_{S'}$ contains $K_{|\bd_GS'|}$ or a dart as a rooted minor as well.
Moreover, we clearly have $|S'|\ge 2$, and $|V(G)\setminus S'|\ge |T\cup\bd_HT|\ge 6$ since $\rhoFour(H,T)>0$.
Finally, we have $\rhoFour(G,S')\le 0$, as otherwise $(S',T)$ would be a dense $(\le\!4)$-bifragment in $G$.
However, then $S'$ is a reducible $(\le\!4)$-fragment in $G$, contradicting Lemma~\ref{lem:den-no-reducible}.

Therefore, we have $z\not\in S$, and symmetrically $z\not\in T$.
Consider any $Z\in \{S,T\}$.
Note that $Z$ is a $(\le\!5)$-fragment in $G$ and that $\rhoFour(G,Z)\ge \rhoFour(H,Z)>0$.
Moreover, in case that $Z$ is a $5$-fragment, we have $z\in \bd_HZ$, and
$\rhoFour(G,Z)=\rhoFour(H,Z)+q$, where $q$ is the number of common neighbors of $u$ and $v$ in $Z$.
It follows that either $\rhoFour(G,Z)\ge 2$, or $\rhoFour(G,Z)=1$ and $u$ and $v$ have no common neighbor in $Z$,
and in either case the $5$-rooted graph $R^G_Z$ is quite heavy.
However, then the $(\le\!5)$-bifragment $(S,T)$ in $G$ contradicts Corollary~\ref{cor:bifragment-constraint}.
\end{proof}

In particular, let us note the following simple consequence.
\begin{corollary}\label{cor:delta}
If $G$ is a minimal \denco{}, then $G$ has minimum degree at least five.
\end{corollary}
\begin{proof}
Consider any vertex $v\in V(G)$.  Note that $v$ is not isolated, since otherwise $G-v$ would
be a counterexample.  Hence $v$ is incident with an edge, and $\deg v\ge 5$ since this edge is contained in at least
four triangles.
\end{proof}

Next, let us show a useful lemma concerning neighborhoods of vertices of small degree.

\begin{lemma}\label{lem:lot-on-five}
Let $G$ be a minimal \denco{}, let $v$ be a vertex of $G$ of degree at most seven,
and let $H$ be the subgraph of $G$ induced by the neighbors of $v$.
For every set $S\subseteq V(H)$ of size five, the $5$-rooted graph $\roots{H}{S}$ contains
$K_5^-$ as a rooted minor.
\end{lemma}
\begin{proof}
Lemma~\ref{lem:one-edge-contract} implies that $H$ has minimum degree at least four.
If $|V(H)|=5$, then $H$ is isomorphic to $K_5$, and thus the claim holds trivially.

If $|V(H)|=6$, then the complement of $H$ has maximum degree at most one.
Hence, $H[S]$ has at most two non-edges.  If $H[S]$ is a clique, then the claim is again
trivial. Otherwise, let $uv$ be a non-edge of $H[S]$, and let $z$ be the unique vertex
in $V(H)\setminus S$.  Since the complement of $H$ has maximum degree at most one,
$z$ is adjacent in $H$ both to $u$ and to $v$.  Hence, by contracting the edge $uz$
we obtain $K_5^-$ as a rooted minor of $\roots{H}{S}$.

Finally, suppose that $|V(H)|=7$.  In this case the complement $\overline{H}$ of $H$
has maximum degree at most two.  In particular, $|E(\overline{H}[S])|\le 5$.
Let $z_1$ and $z_2$ be the two vertices in $V(H)\setminus S$.  Note that if a vertex has degree two in $\overline{H}[S]$,
then it is adjacent to both $z_1$ and $z_2$ in $H$.  Let us distinguish several cases:
\begin{itemize}
\item If $|E(\overline{H}[S])|=5$, then $\overline{H}[S]$ is a $5$-cycle $v_1\ldots v_5$.
We obtain $K_5^-$ as a rooted minor of $\roots{H}{S}$ by contracting the edges $v_1z_1$ and $v_3z_2$.
Hence, we can assume that $|E(\overline{H}[S])|\le 4$.
\item If $\overline{H}[S]$ contains a cycle $v_1\ldots v_k$ (where $k\in \{3,4\}$),
then we again obtain $K_5^-$ as a rooted minor of $\roots{H}{S}$ by contracting the edges $v_1z_1$ and $v_3z_2$.
Therefore, we can assume that $\overline{H}[S]$ is a forest.
\item Hence, if $|E(\overline{H}[S])|=4$, then $\overline{H}[S]$ is a path $v_1\ldots v_5$.
Note that $v_1$ and $v_5$ each have at most one neighbor in $\{z_1,z_2\}$ in $\overline{H}$.
By symmetry, we can assume that $v_1z_1\in E(H)$.  
We obtain $K_5^-$ as a rooted minor of $\roots{H}{S}$ by contracting the edges $v_1z_1$ and $v_3z_2$.
Hence, we can assume that $|E(\overline{H}[S])|\le 3$.
\item If $\overline{H}[S]$ contains a path $v_1v_2v_3$, then note that $v_1$ and $v_3$ each have at most one neighbor
in $\{z_1,z_2\}$ in $\overline{H}$.  Hence, we can obtain $K_5^-$ as a rooted minor of $\roots{H}{S}$
by contracting the edges $v_2z_1$ and $v_2z_2$.  Therefore, we can assume that $\overline{H}[S]$
has maximum degree at most one.
\item If $\overline{H}[S]$ has at most one edge, then the claim is trivial.
Hence, suppose that $\overline{H}[S]$ has exactly two edges $v_1v_2$ and $v_3v_4$.
If there exists $i\in \{1,3\}$ and $j\in\{1,2\}$ such that $v_iz_j,v_{i+1}z_j\in E(H)$,
then we obtain $K_5^-$ as a rooted minor of $\roots{H}{S}$ by contracting the edge $v_iz_j$.

Otherwise, since $\overline{H}$ has maximum degree at most two, it follows that
$\overline{H}$ contains a matching between $\{v_1,v_2\}$ and $\{z_1,z_2\}$
and between $\{v_3,v_4\}$ and $\{z_1,z_2\}$.  Moreover, $\Delta(\overline{H})\le 2$ also implies $z_1z_2\in E(H)$.
By symmetry, we can assume that $v_1z_1,v_2z_2\in E(H)$, and
we obtain $K_5^-$ as a rooted minor of $\roots{H}{S}$ by contracting these two edges.
\end{itemize}
\end{proof}

\noindent By further contracting the edges in $K_5^-$, we have the following consequence.
\begin{corollary}\label{cor:lot-on-lefour}
Let $G$ be a minimal \denco{}, let $v$ be a vertex of $G$ of degree at most seven,
and let $H$ be the subgraph of $G$ induced by the neighbors of $v$.
For every set $S\subseteq V(H)$ of size at most four, the $|S|$-rooted graph $\roots{H}{S}$ contains
$K_{|S|}$ as a rooted minor.
\end{corollary}

Moreover, Lemma~\ref{lem:lot-on-five} implies that small degree vertices must be separated from the
rest of the minimal \denco{} by $(\le\!4)$-cuts.
\begin{lemma}\label{lem:sepby4}
Let $G$ be a minimal \denco{}, let $v$ be a vertex of $G$ of degree at most seven, let $M$ be the
set of its neighbors, and let $C_1$, \ldots, $C_k$ be the vertex sets of the components of $G-(M\cup\{v\})$.
Then $|\bd_GC_i|\le 4$ and $\rhoFour(G,C_i)\le 0$ holds for each $i\in\{1,\ldots,k\}$.
\end{lemma}
\begin{proof}
Let $H=G[M]$ be the subgraph of $G$ induced by the neighbors of $v$ and consider any $i\in\{1,\ldots,k\}$.
Let $S$ be a subset of $\bd_GC_i$ of size $\min(5,|\bd_GC_i|)$.

Suppose first that $|\bd_GC_i|\ge 5$.  By Lemma~\ref{lem:lot-on-five}, we can contract $H$ to $K_5^-$ on $S$.
By further contracting $G[C_i]$ to a single vertex $u$, we obtain $\target$ on $S\cup\{u,v\}$ as a minor of $G$.
This is a contradiction.  Therefore, we have $|\bd_GC_i|\le 4$ and $S=\bd_GC_i$.

Suppose now that $\rhoFour(G,C_i)>0$, and let $Y=V(G)\setminus (C_i\cup S)$. Then $(C_i,Y)$ is a $(\le\!4)$-bifragment
in $G$, and since $G$ is $4$-bilight, we have $\rhoFour(G,Y)\le 0$.  Corollary~\ref{cor:lot-on-lefour} implies
that $K_{|S|}$ is a rooted minor of $R^G_Y$.  Moreover, since $G$ has minimum degree at least five, we have
$|C_i\cup S|\ge 6$.  It follows that the $(\le\!4)$-fragment $Y$ of $G$ is reducible, contradicting Lemma~\ref{lem:den-no-reducible}.
This contradiction shows that $\rhoFour(G,C_i)\le 0$.
\end{proof}

With this, it is easy to further increase the lower bound on the minimum degree of a minimal \denco{}.
\begin{corollary}\label{cor:deg8}
Every minimal \denco{} has minimum degree at least eight.
\end{corollary}
\begin{proof}
Suppose for a contradiction that a minimal \denco{} $G$ contains a vertex $v$ of degree at most $7$.
Let $Q$ be the set consisting of $v$ and all neighbors of $v$ and let $C_1$, \ldots, $C_k$ be the vertex sets of the components of $G-Q$.
By Lemma~\ref{lem:sepby4}, we have
$$-7=\rhoFour(G)=|E(G[Q])|-4|Q|+\sum_{i=1}^k \rhoFour(G,C_i)\le |E(G[Q])|-4|Q|,$$
and thus
$$|E(G[Q])|\ge 4|Q|-7.$$
If $\deg v=5$, then $|Q|=6$ and $|E(G[Q])|\ge 17>\binom{|Q|}{2}$, which is a contradiction.
If $\deg v=6$, then $|Q|=7$ and $|E(G[Q])|\ge 21=\binom{|Q|}{2}$.  Hence, $G[Q]$ is a clique of size seven, a contradiction
since $G$ is $\target$-minor-free.

Therefore, we have $\deg v=7$, $|Q|=8$, and $|E(G[Q])|\ge 25$ while $\binom{|Q|}{2}=28$.
Hence, $G[Q]$ has $m\le 3$ non-edges.  If at least $m-1$ of them are incident with the same vertex $z$,
then $G[Q]-z$ is a supergraph of $K_7^-\supset\target$, which is a contradiction.  Therefore $G[Q]$ has
exactly three non-edges $u_1u_2$, $u_3u_4$, and $u_5u_6$ not incident with the same vertex.
However, then $G[Q]-u_6$ is isomorphic to $\target$, which is a contradiction.
\end{proof}

\noindent The rest of the proof is trivial.
\begin{proof}[Proof of Theorem~\ref{thm:density}]
Suppose for a contradiction that Theorem~\ref{thm:density} is false, and thus
there exists a minimal \denco{} $G$.  Corollary~\ref{cor:deg8} implies that
$|E(G)|\ge 4|V(G)|$.  However, this contradicts Lemma~\ref{lemma:one-edge-delete}.
\end{proof}

\section{6-colorability of $\target$-minor-free graphs}
\label{sec:coloring}

We are now ready to prove our main result, Theorem~\ref{thm:coloring}.
Recall that a graph $G$ is \emph{$k$-contraction-critical} if it has chromatic number~$k$, but every proper minor of $G$ is
$(k-1)$-colorable.  A minimal counterexample to Theorem~\ref{thm:coloring} clearly is $7$-contraction-critical, and thus
it will be convenient to recall a few well-known facts about such graphs.

\begin{theorem}[Mader~\cite{Mader1968Connectivity}]\label{thm:critical-connected}
For every $k\geq7$, every $k$-contraction-critical graph other than $K_k$ is $7$-connected.
\end{theorem}

Theorem~\ref{thm:critical-connected} easily implies that a $7$-contraction-critical graph $G$ other than $K_7$
has clique number at most five.  In our setting, it in fact implies the following stronger claim.
\begin{lemma}\label{lem:no-K6eq}
If a graph $G$ is $7$-contraction-critical and $K_7^-$-minor-free, then it does not contain $K_6^-$ as a subgraph.
\end{lemma}
\begin{proof}
Consider any $6$-vertex set $W\subseteq V(G)$.  Since $G$ is $7$-connected by Theorem~\ref{thm:critical-connected},
the subgraph $G-W$ is connected and every vertex of $W$ has a neighbor in $V(G)\setminus W$.
The minor of $G$ obtained by contracting $G-W$ to a single vertex cannot contain $K_7^-$ as a subgraph,
and thus $G[W]$ does not contain $K_6^-$ as a subgraph.
\end{proof}

On the other hand, the following result forces the neighborhoods of small-degree vertices to be quite dense.

\begin{theorem}[Dirac~\cite{Dirac1960}]\label{thm:alpha}
For every $k$-contraction-critical graph $G$, the subgraph of $G$ induced by the neighbors
of any vertex $v\in V(G)$ has independence number at most $\deg v-k+2$.
In particular if $G\neq K_k$, then $G$ has minimum degree at least $k$.
\end{theorem}

Consequently, the neighborhood of a vertex $v$ of degree seven in a $7$-contract\-ion-critical
graph has independence number at most two, which often forces the presence of a clique of size four (combining with $v$
to a clique of size five).
However, there is an important exception.  The \emph{Moser spindle} is the 7-vertex graph
obtained from a 5-cycle $v_1\ldots v_5$ by blowing up two non-adjacent vertices, say $v_2$ and $v_4$; that is,
by adding a vertex $v'_2$ adjacent to $v_1$, $v_2$, and $v_3$, and a vertex $v'_4$ adjacent to $v_3$, $v_4$, and $v_5$.

\begin{theorem}[{Kawarabayashi and Toft~\cite[Section~2]{KT2005}}]\label{thm:spindle}
Let $H$ be a $7$-vertex graph. If $\alpha(H)\le 2$, then $H$ contains $K_4$ or the Moser spindle
as a subgraph.
\end{theorem}

We handle the exceptional Moser spindle case using a result of Kriesell and Mohr~\cite{KM2019}
on minor-realizability of Kempe chains.  Given a proper coloring $\varphi$ of a graph $G$,
we say that differently colored vertices $u,v\in V(G)$ are \emph{$\varphi$-Kempe-adjacent} if $G$ contains a path from $u$ to $v$
on which $\varphi$ only uses the colors $\varphi(u)$ and $\varphi(v)$.

\begin{theorem}[{Kriesell and Mohr~\cite[Lemma~2]{KM2019}}]\label{thm:Kempe-cycle}
Let $\varphi$ be a proper coloring of a graph $G$ and let $v_1$, \ldots, $v_k$
be vertices of $G$ given pairwise distinct colors by $\varphi$.
If the vertices $v_i$ and $v_{i+1}$ are $\varphi$-Kempe-adjacent for every $i\in\{1,\ldots, k\}$ (where $v_{k+1}=v_1$),
then $\roots{G}{\{v_1,\ldots,v_k\}}$ contains the $k$-cycle $v_1\ldots v_k$
as an $\id$-rooted minor.
\end{theorem}

With this, we can further restrain vertices of degree seven.

\begin{lemma}\label{lem:seven-in-K5}
Let $G$ be a $7$-contraction-critical graph.  If $G$ is $K_7^-$-minor-free, then every vertex
of $G$ of degree seven is contained in a clique of size five.
\end{lemma}
\begin{proof}
Consider any vertex $v\in V(G)$ of degree seven and let $H$ be the subgraph of $G$ induced by its neighbors.
By Theorem~\ref{thm:alpha}, we have $\alpha(H)\le 2$, and thus Theorem~\ref{thm:spindle}
implies that $H$ contains $K_4$ or the Moser spindle as a subgraph.  If $H$ contains $K_4$
as a subgraph, then $v$ is contained in a clique of size five.

Suppose for a contradiction that this is not the case.  Hence, we can label the vertices of $H$
as $v_1$, \ldots, $v_5$, $v'_2$, and $v'_4$ so that $v_1\ldots v_5$ is a cycle, $v'_2$ is adjacent to $v_1$, $v_2$, and $v_3$,
and $v'_4$ is adjacent to $v_3$, $v_4$, and $v_5$.  Since $K_4\not\subseteq H$, we have $v_1v_3\not\in E(H)$.
Let $G'$ be the minor of $G$ obtained by contracting the path $v_1vv_3$ to a single vertex $u$.
Since $G$ is $7$-contraction-critical, the graph $G'$ is $6$-colorable.  Let $\varphi$ be the $6$-coloring of $G-v$
obtained from a $6$-coloring of $G'$ by giving both $v_1$ and $v_3$ the color of $u$.
We can assume that $\varphi(v_1)=\varphi(v_3)=6$.  Note that $\varphi$ has to use the colors $1$, \ldots, $5$
on $V(H)$ as well, as otherwise we could extend $\varphi$ to a $6$-coloring of $G$ by giving $v$ a color not used on $V(H)$.
Hence, we can furthermore assume that $\varphi(v'_2)=1$, $\varphi(v_2)=2$, $\varphi(v'_4)=3$, $\varphi(v_4)=4$, and $\varphi(v_5)=5$.

Suppose now that there exist distinct vertices $x,y\in V(H)\setminus\{v_1,v_3\}$ such that $x$ and $y$ are not $\varphi$-Kempe-adjacent.
Equivalently, the component $C$ of $G[\varphi^{-1}(\{\varphi(x),\varphi(y)\})]$ that contains $x$ does not contain $y$.
Hence, we can obtain a $6$-coloring of $G$ from $\varphi$ by exchanging the colors $\varphi(x)$ and $\varphi(y)$ on $C$
(thus recoloring $x$ to $\varphi(y)$), then giving $v$ the color $\varphi(x)$.  This is a contradiction, and thus
the vertices $v_2$, $v'_2$, $v_4$, $v'_4$, and $v_5$ are pairwise $\varphi$-Kempe-adjacent.

Let $\psi$ be the restriction of $\varphi$ to the graph $G''=G[\varphi^{-1}(\{1,\ldots,4\})]$.
Note that the vertices $v_2$, $v'_2$, $v_4$, and $v'_4$ of $G''$ are also pairwise $\psi$-Kempe-adjacent.
By Theorem~\ref{thm:Kempe-cycle}, the rooted graph $\roots{G''}{\{v_2,v'_2,v_4,v'_4\}}$ contains the 4-cycle
$v_2v_4v'_2v'_4$ as an $\id$-rooted minor.  Since we also have $v_2v'_2, v_4v'_4\in E(G'')$, this
rooted graph actually contains $K_4$ as an $\id$-rooted minor.  By contracting $G''$ to $K_4$ on $\{v_2,v'_2,v_4,v'_4\}$
and additionally contracting the edge $v_1v_5$, we see that $\roots{G-v}{V(H)\setminus\{v_5\}}$ contains $K_6^-$
as a rooted minor, and thus $G$ contains $K_7^-$ as a minor.  This is a contradiction.
\end{proof}

Moreover, we can almost eliminate cliques of size five.
\begin{lemma}\label{lem:K5inters}
Every $K_7^-$-minor-free $7$-contraction-critical graph $G$ contains at most one clique of size five.
\end{lemma}
\begin{proof}
Suppose for a contradiction that there exist two distinct sets $L_1,L_2\subset V(G)$ of size five inducing cliques in $G$.
Let $k=|L_1\cap L_2|$.  By Lemma~\ref{lem:no-K6eq}, we have $k\le 3$.
Since $G$ is $7$-connected by Theorem~\ref{thm:critical-connected}, Menger's theorem implies that $G$ contains
five pairwise vertex-disjoint paths $P_1$, \ldots, $P_5$ from $L_1$ to $L_2$.  We can choose the labels so that
$P_1$, \ldots, $P_k$ are the single-vertex paths covering the intersection $L_1\cap L_2$.

Let $u$ and $v$ be the ends of the path $P_5$, where $u\in L_1$ and $v\in L_2$, and let $Z=(L_1\setminus \{u\})\cup \{v\}$.
Since $G$ is $7$-connected, the graph $G-Z$ is connected and contains a path $Q_0$ from $u$ to $L_2\setminus Z$.
Since $(L_1\cap L_2)\cup\{v\}\subset Z$, the end of $Q_0$ in $L_2$ lies on one of the paths $P_{k+1}$, \ldots, $P_4$.
Let $Q$ be a minimal segment of $Q$ with one end on $P_5$ and the other end in $V(P_{k+1}\cup \cdots\cup P_4)$.
Without loss of generality, we can assume that $Q$ ends on $P_4$.  Let $u'$ be the end of $P_4$ in $L_1$.
Note that since $u',v\in Z$, we have $u',v\not\in V(Q)$.

We now contract each of the paths $P_{k+1}$, \ldots, $P_3$, $P_4-u'$, and $P_5-v$ to a single vertex,
then contact the path $Q$ to a single edge.  This way, we obtain $K_7^-$ as a minor of $G$, which is a contradiction.
\end{proof}

\noindent This gives us Theorem~\ref{thm:counterdens}.
\begin{proof}[Proof of Theorem~\ref{thm:counterdens}]
Let $v$ be an arbitrary vertex of $G$.  The proper minor $G-v$ of $G$ is $6$-colorable by the assumptions, and thus
the graph $G$ is $7$-colorable.  Therefore, $G$ has chromatic number exactly seven, and consequently it is $7$-contraction-critical.
Since $K_7^-$ is not a minor of $G$, we have $G\neq K_7$, and thus Theorem~\ref{thm:critical-connected} implies that $G$
is $7$-connected.  In particular, $G$ has minimum degree at least seven.

Moreover, $G$ has at most one clique of size five by Lemma~\ref{lem:K5inters},
and every vertex of $G$ of degree seven is contained in it by
Lemma~\ref{lem:seven-in-K5}.  Therefore, $G$ has at most five vertices of
degree seven.  Consequently,
$$|E(G)|\ge \Bigl\lceil\frac{8|V(G)|-5}{2}\Bigr\rceil=4|V(G)|-2.$$
\end{proof}

\noindent Showing that $\target$-minor-free graphs are $6$-colorable is now trivial.

\begin{proof}[Proof of Theorem~\ref{thm:coloring}]
Suppose for a contradiction that there exists a $\target$-minor-free graph $G$ of chromatic number at least $7$,
and choose one with $|V(G)|+|E(G)|$ minimal.  It follows that every proper minor of $G$ is $6$-colorable.
Since $G$ clearly does not contain $K_7^-\supset\target$ as a minor, Theorem~\ref{thm:counterdens}
implies that $G$ is $7$-connected and $|E(G)|\ge 4|V(G)|-2$.
However, Theorem~\ref{thm:denstarget} then implies that $G$ contains $\target$ as a minor, which is a contradiction.
\end{proof}

\bibliographystyle{alpha}
\bibliography{main}

\end{document}